\documentclass[reqno, 12pt]{amsart}
\usepackage{amsmath}
\usepackage{amssymb}
\usepackage{amsthm}
\usepackage{graphicx}
\newcommand\NoBlackBoxes{\global\overfullrule0pt}
\NoBlackBoxes
\theoremstyle{plain} 
\usepackage{color}              
\usepackage{epsfig}
\newtheorem{theorem}{Theorem}[section]
\newtheorem{lemma}[theorem]{Lemma}
\newtheorem{corollary}[theorem]{Corollary}
\newtheorem{proposition}[theorem]{Proposition}

\newtheorem{remark}[theorem]{Remark}

\numberwithin{equation}{section}

\renewcommand\Re{\operatorname{Re}}
\renewcommand\Im{\operatorname{Im}}

\newcommand{\Cal}{\mathcal}

\newcommand{\tc}{\textcolor{red}}

\begin{document}
\title{Characterization problems for linear forms with free summands.} 

\author{G. P. Chistyakov$^{1,2,3}$}
\thanks{1) Faculty of Mathematics, 
University of Bielefeld, Germany} 
\thanks{2) Institute for Low Temperature Physics and Engineering, 
Kharkov, Ukraine}
\thanks{3) Research supported by SFB 701}
\address
{Gennadii Chistyakov \newline
Fakult\"at f\"ur Mathematik\newline
Universit\"at Bielefeld\newline
Postfach 100131\newline
33501 Bielefeld \newline
Germany}
\email {chistyak@math.uni-bielefeld.de}

\author{F. G\"otze$^{1,3}$}
\address
{Friedrich G\"otze\newline
Fakult\"at f\"ur Mathematik\newline
Universit\"at Bielefeld\newline
Postfach 100131\newline
33501 Bielefeld \newline
Germany}
\email {goetze@math.uni-bielefeld.de}



\date{October 2011}


\leftmark
\rightmark
\begin{abstract}
Let $T_1,\dots,T_n$ denote free random variables. 
For two linear forms $L_1=\sum_{j=1}^n a_jT_j$ and $L_2=\sum_{j=1}^n b_jT_j$ 
with real coefficients $a_j$ and $b_j$
we shall describe all distributions of $T_1,\dots,T_n$ such that $L_1$ and $L_2$ are free. 
For identically distributed free random variables $T_1,\dots,T_n$ with 
distribution $\mu$ we establish necessary and sufficient 
conditions on the~coefficients $a_j,b_j,\,j=1,\dots,n,$ such that 
the~statements:\\ $(i)$ $\mu$ is a~centered semicircular distribution; and 
 $(ii)$ \, $L_1$ and $L_2$ are identically distributed ($L_1\stackrel{D}{=}L_2$); 
are equivalent. In the proof we give  a complete characterization of all sequences 
of free cumulants of  measures with compact support and with a finite number of non zero entries.
The characterization is based on  topological properties of regions defined by means of 
the Voiculescu transform $\phi$ of such sequences.

\end{abstract}

\maketitle

\section{Introduction} 
The~intensive research on the~asymptotic behaviour of random matrices
induced  more research on their infinitely dimensional limiting models as well. 
Free convolution of probability measures, 
introduced by D. Voiculescu, may be regarded as such a model~\cite{Vo:1986},
\cite{Vo:1987}.
The~key concept of this definition is the~notion of freeness,
which can be interpreted as a~kind of independence for non-commutative
random variables. As in the~classical probability the~concept of
independence gives rise to the~classical convolution, the~concept
of freeness leads to a~binary operation on the~probability measures
on the~real line, the~free convolution. Many classical results
in the~theory of addition of independent random variables have their
counterpart in this theory, such as the~law of large numbers,
the~central limit theorem, the~L\'evy-Khintchine formula and others.
We refer to Voiculescu, Dykema and Nica~\cite{Vo:1992}, Hiai and Petz~\cite{HiPe:2000},
and Nica and Speicher~\cite{NiSp:2006} for an~introduction to these topics.

In many problems of mathematical statistics, conclusions are based on
the~fact that certain special distributions have important
properties which permit the reduction of the~original problem to
a~substantially simpler one, for instance via the~notion of sufficiency. 

The~simplest type of statistics of independent observations, admitting 
a~fairly complete description of the~mutual independence and identical distribution,
are linear statistics. 

Consider independent scalar random variables $X_1,\dots, X_n$ (not necessary identically
distributed) and two linear statistics 
\begin{equation}\label{1.1}
L_1:=\sum_{j=1}^n\alpha_j X_j\quad\text{and}\quad L_2:=\sum_{j=1}^n\beta_j X_j,
\end{equation}
where $\alpha_j,\,\beta_j$ are real constant coefficients. It turns out that the~independence
of the~two linear statistics $L_1$ and $L_2$ essentially characterizes 
the~normality of the~variables $X_j$.
To be precise, the~following assertion, due to Darmois~\cite{Dar:1953} and 
Skitovich~\cite{Sk:1953}, \cite{Sk:1954}, holds.

Let $L_1$ and $L_2$ given by (\ref{1.1}) be independent. Then the~random variables
$X_j$ such that $\alpha_j\beta_j\ne 0$, i.e., which enter in both $L_1$ and $L_2$, have normal
distributions.

Note that the~converse proposition holds in the~following form: if $\sum_{j=1}^n \alpha_j
\beta_j Var( X_j)=0$ and all $X_j$ such that $\alpha_j\beta_j\ne 0$ are normal,
then $L_1$ and $L_2$ are independent.

Polya~\cite{Po:1923} was the~first who established that only the~normal distribution
leads to identically distributed linear statistics $X_1$ and $a_1X_1+a_2X_2$, where $X_1$ and $X_2$
are independent and identically distributed. Marcinkiewicz~\cite{Mar:1938} proved that 
distributions having moments of all orders and admitting the~existence of
a~ nontrivial pair of identically distributed linear statistics based on 
a~random sample are normal. Yu. V.~Linnik~\cite{Li:1963}, \cite{Li:1962} described the~class
of symmetric distributions admitting identically distributed linear
statistics and studied in detail the~problem of characterizing 
the~normal distribution via properties of such statistics.

In this paper we give a~complete description of those free random variables
$T_1,\dots, T_n$ such that the~linear statistics $a_1T_1+\dots a_nT_n$ and 
$b_1T_1+\dots b_nT_n$ are free.

In addition we prove an~analogue of Yu. V.~Linnik's results
\cite{Li:1963}, \cite{Li:1962}, \cite{KLR:1973}, 
and give the~solution of the~problem of characterization of the~semicircular
distribution via identical distribution of linear
statistics $a_1T_1+\dots a_nT_n$ and $b_1T_1+\dots b_nT_n$, where
$T_1,\dots,T_n$ are identically distributed free random variables.

\section{Results}

Assume that $A$ is a~finite von Neumann algebra with normal faithful
trace state $\tau$. The~pair $(A,\tau)$ will be called a~{\it tracial}
$W^*$-{\it probability space}. Assume that $A$ is acting on a~Hilbert
space $H$. We will denote  by $\tilde A$ the~set of all operators
on $H$ which are affiliated with $A$ and by $\tilde{A}_{sa}$ the~set
of selfadjoint operators affiliated with $A$. Recall that a (generally unbounded) selfadjoint
operator $X$ on $H$ is called {\it affiliated} with $A$ if all the~spectral projections
of $X$ belong to $A$. The~elements of $\tilde{A}_{sa}$ will be regarded
as (possibly) unbounded random variables. Let $T\in \tilde{A}_{sa}$.
The~distribution $\mu_{T}$ of $T$ is the~unique  probability measure 
on $\Bbb R$ satisfying the~equality
$$
\tau(u(T))=\int\limits_{\Bbb R}u(\lambda)\,\mu_T(d\lambda)
$$
for every bounded Borel function $u$ on $\mathbb R$. 

Recall that a~family $\{T_j\}_{j=1}^k$ of elements of $T\in \tilde{A}_{sa}$ is said
to be free if for all bounded continuous functions $u_1,u_2,\dots,u_n$
on $\mathbb R$ we have $\tau(u_1(T_{j_1})u_2(T_{j_2})\dots u_n(T_{j_n}))=0$
whenever $\tau(u_l(T_{j_l}))=0,\,l=1,\dots,n$, and all alternating
sequences $j_1,j_2,\dots,j_n$ of $1$'s, $2$'s, and $k$'s, i.e., $j_1\ne j_2\ne
\dots\ne j_n$. 

Bercovici and 
Voiculescu~\cite{BeVo:1993} proved that if $T_j\in\tilde{A}_{sa}$ are
free random variables for $j=1,\dots,n$, and $Q$ is a~selfadjoint 
polynomial in $n$ non-commuting variables, then the~distribution of
the~random variable $Q(T_1,T_2,\dots,T_n)$ depends only on 
the~distributions  of $T_1,T_2,\dots,T_n$.

If $Q(T_1,\dots,T_n)=T_1+\dots+T_n$, then the~distribution of $T_1+\dots+T_n$ only depends
on the~$\mu_{T_j}$ and is called the~{\it additive free convolution} of
$\mu_{T_1},\dots,\mu_{T_n}$. Denote this distribution by $\mu_{T_1}\boxplus\dots\boxplus
\mu_{T_n}$.

Let $T_1,\dots,T_n$ denote free random variables 
with distributions $\mu_1,\dots,\mu_n$, respectively. Consider two linear statistics 
\begin{equation}\label{2.1}
L_1:=a_1T_1+\dots+a_nT_n\quad\text{ and}\quad L_2:=b_1T_1+\dots+b_nT_n
\end{equation} 
with real coefficients $a_j$ and $b_j$.
In the~sequel we assume without loss of generality
that $|a_j|\le 1$ and $|b_j|\le 1$ for all $j=1,\dots,n$.

Denote by $\mu_w$ the~{\it standard} semicircular measure, i.e.,
the~measure with the~density $p_{w}(x)=\frac 1{2\pi}\sqrt{(4-x^2)_+}$,
where $a_+=\max\{0,a\}$. We shall call measures with densities
$p(x)=\frac 1{2\pi a^2}\sqrt{(4a^2-(x-b)^2)_+}$ with some $a>0$ and $b\in\mathbb R$ semicircular.

Nica~\cite{Ni:1996} established that the~stability of freeness under rotations characterizes
semicircular random variables. 
Lehner~\cite{Leh:2004} proved that there are free random variables $T_1,T_2,T_3$
which are not semicircular and such that $L_1:=a_1T_1+a_2T_2+a_3T_3$ and
$L_2:=b_1T_1+b_2T_2+b_3T_3$ are free. Hence the~analogue of the~Darmois--Skitovich
theorem fails in the~free case if there are at least three random variables involved.
We can nevertheless describe all free random variables $T_1,\dots,T_n$
for which the~linear statistics $L_1$ and $L_2$ in (\ref{2.1}) are free. Our result
extends the~results of Nica and Lehner considerably.
In order to formulate our first result we need the~following notation.

Let $T\in\tilde{A}_{as}$ be a~given random variable with distribution $\mu$
such that $\beta_n(\mu):=\int_{\mathbb R}|x|^n\,\mu(dx)<\infty$ for some $n\in\mathbb N$
and let $T^{(k)},\,k=1,\dots,n$, 
be its free copies. Let $\omega$ be $n$-th primitive root of unity (e.g., 
$\omega=e^{2\pi i/n}$) and set
\begin{equation}\label{2.2}
T^{\omega}=\omega T^{(1)}+\omega^2 T^{(2)}+\dots+\omega^n T^{(n)}.
\end{equation} 
Following Lehner~\cite{Leh:2004}, we define the~$n$th free cumulant of the~random variable $T$
to be
\begin{equation}\label{2.3}
\kappa_n(T)=\frac 1n\tau((T^{\omega})^n)
\end{equation} 
in a~short way. 
For a detailed definition of free cumulants 
see the monograph of Nica and Speicher~\cite{NiSp:2006} as well.
Since the free cumulant $\kappa_n$ depends on $n$ and the~distribution $\mu$ of $T$ only, we will denote
this cumulant by $\kappa_n(\mu)$ as well.

\begin{theorem}\label{2.1ath}
Consider free random variables $T_1,\dots,T_n,\,n\ge 2$, and let
$a_j,\,b_j$ be real numbers such that $a_jb_j\ne 0$ and $\frac{b_j}{a_j}\ne \frac{b_s}{a_s}$
for $j,s=1,\dots,m$, where $m\le n$, and $a_jb_j=0$ for $j=m+1,\dots,n$. The~linear statistics 
$L_1$ and $L_2$ are free
if and only if the~distributions $\mu_1,\dots,\mu_m$ have compact supports
and the~free cumulants $\kappa_s(T_j),\,j=1,\dots,m$, satisfy the~relations:
\begin{align}\label{2.4}
\sum_{j=1}^m a_j^lb_j^t\kappa_s(T_j)=0
\end{align} 
for all $s=2,\dots,m$ and $(l,t)\in\mathbb N^2$ such that $l+t=s$,
and $\kappa_s(T_j)=0$ for $s\ge m+1$.
\end{theorem}

The~following result describes all distributions $\mu_1,\dots,\mu_m$ in the~previous theorem.
Let $\kappa_1,\dots,\kappa_m$ be real numbers. Introduce the function
\begin{equation}\notag
\varphi(z):=\kappa_1+\frac{\kappa_2}{z}+\dots+\frac{\kappa_m}{z^{m-1}},\quad z\in\mathbb C\setminus\{0\}. 
\end{equation}
Denote by $\Omega_{\varphi}$ 
the component of $\{z\in\mathbb C^+:\Im (z+\varphi(z))>0\}$
which contains $\infty$. 
\begin{theorem}\label{2.2ath}
A~sequence $\{\kappa_n\}_{n=1}^{\infty}$ of real numbers such that $\kappa_n=0$ for
$n\ge m+1,\,m\ge 2$, is a~sequence of free cumulants of some probability measure with compact support 
if and only if every Jordan curve, contained in $\mathbb C^+\cup\mathbb R$ and connecting $0$ and $\infty$,
contains a point of the boundary of $\Omega_{\varphi}$.
\end{theorem}
\begin{remark}\label{2.2arem}
Consider the set $S$ of free cumulant sequences of the form $\{\kappa_1,\kappa_2,\dots,\kappa_m,0,0\dots\}$.
Then the set $\{\kappa_1,\kappa_2,\dots,\kappa_m\}$ is closed in 
the space $\mathbb R^{m}$. 
\end{remark}

We easily obtain from Theorem~\ref{2.2ath} the following consequence.
\begin{corollary}\label{2.3acor} 
A~sequence $\{\kappa_n\}_{n=1}^{\infty}$ of real numbers such that 
$\kappa_2>0$, $\kappa_n=0$ for $n\ge m+1,\,m\ge 2$, and $|\kappa_n|\le \varepsilon,\,n=3,\dots,m$,
with sufficiently small $\varepsilon>0$,
is a~sequence of free cumulants of some probability measure with compact support. 
\end{corollary}  
Note that Bercovici and 
Voiculescu \cite{BeVo:1995} proved a~more general result than Corollary~\ref{2.3acor}
and showed the~failure of the~well-known Cram\'er and Marcinkiewicz theorems
in free probability theory.
To illustrate these results in a low dimensional example we consider the case $m=4$.
\begin{corollary}\label{2.2acor} 
A~sequence $0,1,\kappa_3,\kappa_4,0,0,\dots$ is a~free cumulant sequence of some probability measure
if and only if $(\kappa_3,\kappa_4)\in D$, where 
\begin{equation}
D:=\Big\{(x,y)\in\mathbb R^2:|x|\le f_1(y),
\,-\frac 1{12}\le y\le \frac 1{36}\Big\}\cup
\Big\{(x,y)\in\mathbb R^2:|x|\le f_2(y),
\,\frac 1{36}<y\le \frac 14\Big\}\notag
\end{equation} 
with
$$f_1(y):=\frac 1{3\sqrt 6}\sqrt{1+\sqrt{1-36y}}\,
\Big(2-\sqrt{1-36y}\Big)\quad \text{for}\,\,\,-\frac 1{12}\le y\le \frac 1{36}; \quad\text{and}$$ 
$$f_2(y) :=\frac{\sqrt 2\root 4\of y}{3\sqrt 3}\frac{(1+\sqrt{1-12\sqrt y+36y})
(2-\sqrt{1-12\sqrt y+36y})}{\sqrt{1-2\sqrt y}}
\quad \text{for}\,\,\frac 1{36}<y<\frac 14; \, f_2\Big(\frac 14\Big) :=0.  $$ 
\end{corollary}
\begin{figure}[htb]
\includegraphics[scale=0.75] {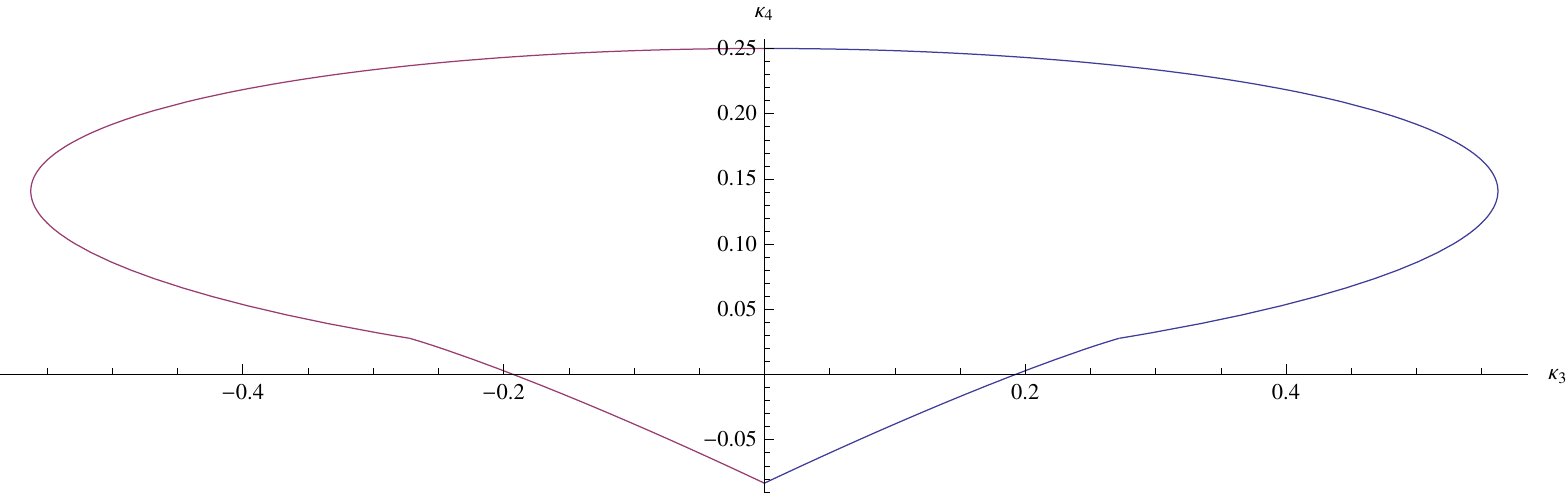}
\caption{\small Region D of realized cumulants $(\kappa_3,\kappa_4)$.}
\label{figure}
\end{figure}
We see from Corollary~\ref{2.2acor} that a~sequence $0,1,\kappa_3,0,\dots$ is 
a~free cumulant sequence of some probability measure
if and only if $|\kappa_3|\le \frac 1{3\sqrt 3}$. 
This assertion was obtained by Lehner (oral communication) by other means.

Now we consider the~problem of the~description of identically distributed free
random variables $T_1,\dots,T_n$ with distribution $\mu$
such that the~statistics $L_1$ and $L_2$ are identically
distributed as well $(L_1\stackrel{D}{=} L_2)$.

Following Linnik~\cite{Li:1963}, we introduce two entire functions of 
a~complex variable $z$:
$$
\Lambda_1(z)=|a_1|^z+\dots+|a_n|^z-|b_1|^z-\dots-|b_n|^z
$$
and
$$
\Lambda_2(z)=a_1^z+\dots+a_n^z-b_1^z-\dots-b_n^z,
$$
where $a_1,\dots,a_n$ and $b_1,\dots, b_n$ are restricted as in (\ref{2.1}).
It is easy to see that all zeros of the~functions $\Lambda_1(z)$ 
lie in a~strip $b_1<\Re z< b_2$ with some $b_1,b_2\in\Bbb R$.

We prove the~following characterization of semicircular measures
which is an~analogue of Linnik's result~\cite{Li:1963} about a~characterization 
of Gaussian probability measures. Recall that a~probability measure $\mu$ is called degenerate if 
$\mu=\delta_a$, where $\delta_a$ is the~Dirac measure concentrated at the~point $a\in\mathbb R$.

\begin{theorem}\label{2.1th}
Let $\Lambda_1(z)\not\equiv 0$.  
In order that, for some non-degenerate probability measure $\mu$, 
the~statement 
$(1)$ $L_1\stackrel{D}{=} L_2$ 
implies the statement
$(2)$ $\mu$ is a~semicircular measure, 
it is necessary and sufficient that the~following conditions are satisfied:  

$\,\,\,$$a)$ $2$ is a~simple and unique positive zero 
of the~function $\Lambda_1(z)$, 

$\,\,\,$$b)$  $\Lambda_2(2m+1)\ne 0$ for all $m=1,\dots$.
\end{theorem}

Note that $(2)$ implies $(1)$ if $\Lambda_2(1)=\Lambda_1(2)=0$. Let $\mu$ be a~semicircular measure
with mean zero, then $(2)$ implies $(1)$ iff $\Lambda_1(2)=0$.

From Theorem~\ref{2.1th} we obviously obtain the~following consequences.
\begin{corollary}\label{2.1^{**}co}
Let $\Lambda_1(z)\not\equiv 0,\,\Lambda_2(1)=0$ and let $\Lambda_1(1)\ne\Lambda_2(1)$. 
In order that, for some non-degenerate probability measure $\mu$,
the statements $(1)$ and $(2)$ are equivalent, it is necessary and sufficient that the~conditions $a)$ and $b)$
are satisfied.
\end{corollary}

\begin{corollary}\label{2.1^*co}
Let $\Lambda_1(z)\not\equiv 0$. In order that, 
for some non-degenerate probability measure $\mu$ 
with a median  equal to $0$,
the statements $(1)$ and $(2)$ are equivalent, it is necessary and sufficient that the~conditions $a)$ and $b)$
are satisfied.
\end{corollary}


Moreover Theorem~\ref{2.1th} implies the following assertion.
\begin{corollary}\label{2.1co}
Assume that $a_1^2+\dots+a_n^2=1$, and $b_1=1,\,b_2=\dots=b_n=0$. 
Furthermore, assume $L_1\stackrel{D}{=} L_2$.
Then either $\mu$ is a~semicircular measure or $\mu$ is a~degenerate probability measure.
\end{corollary}

Corollary~\ref{2.1co} is an~analogue of a~result by Polya~\cite{Po:1923}. For analogues
of the~Polya result in non-commutative probability theory see Lehner~\cite{Leh:2003}. 

We prove as well the following result for symmetric probability measures.
\begin{theorem}\label{2.1th*}
Let $\Lambda_1(z)\not\equiv 0$.  
In order that, for some non-degenerate probability measure $\mu$, 
which is symmetric with respect to $0$, 
the statements $(1)$ and $(2)$ are equivalent, it is necessary and sufficient that the~condition $a)$ 
is satisfied. 
\end{theorem}

The~following result for probability measures $\mu$ 
with moments of finite order is an~analogue of a~result by Linnik~\cite{Li:1963}
in classical probability theory.  
\begin{theorem}\label{2.2th}
Assume that $\Lambda_1(z)\not\equiv 0$ and that $\Lambda_1(z)$ has zeros in $-i\mathbb C^+$. 
Let $\gamma$ denote the~maximum of the~real parts of such zeros.
In order that, for some non-degenerate probability measure $\mu$ such that
$\int\limits_{\Bbb R}u^{2s}\,\mu(du)<\infty$ with $s=[\gamma/2+1]$,
the~statement $(1)$
implies the statement $(2)$ 
it is necessary and sufficient that $\Lambda_2(2)=0$
and $\Lambda_2(m)\ne 0$ for all positive integers $m>2$. 
\end{theorem}

We prove in Lemma~\ref{4.1lem} that if $L_1\stackrel{D}{=} L_2$, 
then the~function $\Lambda_1(z)$ has a~real root $\gamma$ such that
$0<\gamma\le 2$.

In the~case where all moments of $\mu$ exist we obtain from Theorem~\ref{2.2th} the~following result.
\begin{theorem}\label{2.3th}
Assume that $\Lambda_1(z)\not\equiv 0$. 
In order that, for some non-degenerate probability measure $\mu$ such that
$\int\limits_{\Bbb R}u^{2m}\,\mu(du)<\infty$ for all $m\in\mathbb N$, the~statement 
$(1)$ implies the~statement $(2)$ 
it is necessary and sufficient that $\Lambda_2(2)=0$
and $\Lambda_2(m)\ne 0$ for all positive integers $m>2$.
\end{theorem}

Theorem~\ref{2.3th} is an~analogue of a~result by Marcinkiewicz~\cite{Mar:1938} in
classical probability theory.

\section{An~analytic approach to a~solution of the~considered problems. Auxiliary results.}

Denote by $\mathcal M$ the~family of all Borel probability measures
defined on the~real line $\Bbb R$. On the~set $\mathcal M$ 
define two associative composition laws  denoted $*$ and $\boxplus$.
Let $\mu_1,\mu_2\in\mathcal M$. The~measure $\mu_1*\mu_2$ will denote the~classical 
convolution of $\mu_1$ and $\mu_2$. In probabilistic terms, $\mu_1*\mu_2$
is the~probability distribution of $X+Y$, where $X$ and $Y$ are
(commuting) independent random variables with distributions $\mu_1$ 
and $\mu_2$ respectively. The~measure $\mu_1\boxplus\mu_2$ is the~free
(additive) convolution of $\mu_1$ and $\mu_2$ introduced by 
Voiculescu~\cite{Vo:1986}
for compactly supported measures. The~free convolution was extended by
Maassen~\cite{Ma:1992} to measures with finite variance and by Bercovici
and Voiculescu~\cite{BeVo:1993} to the~whole class $\mathcal M$.
Thus, $\mu_1\boxplus\mu_2$ is the~distribution of $X+Y$,
where $X$ and $Y$ are free random variables with distributions $\mu_1$ 
and $\mu_2$, respectively.

Let $\Bbb C^+\,(\Bbb C^-)$ denote the~open upper (lower) half of
the~complex plane. If $\mu\in\mathcal M$, denote its Cauchy transform by
\begin{equation}\label{2.1*}
G_{\mu}(z)=\int\limits_{-\infty}^{\infty}\frac {\mu(dt)}{z-t},
\qquad z\in\Bbb C^+.
\end{equation}

Following Maassen~\cite{Ma:1992} and Bercovici and 
Voiculescu~\cite{BeVo:1993},
we introduce the~{\it reciprocal Cauchy transform}
\begin{equation}\label{2.2*}
F_{\mu}(z)=\frac 1{G_{\mu}(z)}.
\end{equation}
The~corresponding class of reciprocal Cauchy
transforms of all $\mu\in\mathcal M$ we denote by $\mathcal F$.
This class admits a~simple description. Recall that the~Nevanlinna
class $\mathcal N$ is the~class  of analytic functions 
$F:\Bbb C^+\to\Bbb C^+$. The~class $\mathcal F$ is
the~subclass of Nevanlinna's functions $F_{\mu}$ such that 
$F_{\mu}(z)/z\to 1$ as $z\to \infty$ non-tangentially to 
$\mathbb R$ (i.e., such that $\Re z/\Im z$ stays bounded), 
and this implies that $F_{\mu}$ has certain invertibility properties.
(For details see Akhiezer and Glazman~\cite{AkhG:1963}, 
Akhiezer~\cite {Akh:1965}).
To be precise, for two numbers $\alpha>0,\beta>0$
we set
$$
\Gamma_{\alpha}=\{z=x+iy\in\Bbb C^+:|x|<\alpha y\}\quad\text{and}\quad
\Gamma_{\alpha,\beta}=\{z=x+iy\in\Gamma_{\alpha}:y>\beta\}.
$$
Then for every $\alpha>0$ there exists $\beta=\beta(\mu,\alpha)$ 
such that $F_{\mu}$ has the~right inverse $F_{\mu}^{(-1)}$ defined on 
$\Gamma_{\alpha,\beta}$.
The~function $\phi_{\mu}(z)=F_{\mu}^{(-1)}(z)-z$ will be called
the~Voiculescu transform of $\mu$. It is not hard to show that
$\Im \phi_{\mu}(z)\le 0$ for $z\in \Gamma_{\alpha,\beta}$ where 
$\phi_{\mu}$ is defined. Note that $\phi_{\mu}(z)=o(z)$
as $|z|\to\infty$, $z\in\Gamma_{\eta}$. In the~sequel we will denote $\phi_{\mu}(z)$ 
by $\phi_{T}(z)$ for a~random variable $T$ with a~distribution $\mu$ as well.
It is easy to verify that $\phi_{u T}(z)=u\phi_T(z/u)$ 
for fixed $u\in\mathbb R$ and $z\in\Gamma_{\alpha,\beta}$, where
$\phi_{u T}(z)$ and $\phi_{T}(z/u)$ are defined.

In the~domain $\Gamma_{\alpha,\beta}$, where the~functions $\phi_{\mu_1}(z)$, $\phi_{\mu_2}(z)$,
and $\phi_{\mu_1\boxplus\mu_2}(z)$ are defined, we have
\begin{equation}\label{2.3*}
\phi_{\mu_1\boxplus\mu_2}(z)=\phi_{\mu_1}(z)+\phi_{\mu_2}(z).
\end{equation} 
This characterization for the~distribution $\mu_1\boxplus\mu_2$ of $X+Y$, where $X$ and $Y$
are free random variables, is due to Voiculescu~\cite{Vo:1986}. 
He considered compactly supported measures
$\mu$. The~result was extended by Maassen~\cite{Ma:1992} to measures 
with finite variance; the~general case was proved by 
Bercovici and Voiculescu~\cite{BeVo:1993}.

We need the~following auxiliary results.

The~following proposition is proved in~\cite{BeP:1999}.
\begin{proposition}\label{3.4pr}
For every probability measure $\mu$ we have
$$
\phi_{\mu}(z)=z^2\Big(G_{\mu}(z)-\frac 1z\Big)(1+q_{\mu}(z)),
\quad z\in\Gamma_{\alpha,\beta},
$$
where $q_{\mu}(z)=o(1)$ as $z\to\infty$.
\end{proposition}

The~following lemma is well-known, see~\cite{Akh:1965}. 
\begin{lemma}\label{3.4abl}
Let $\mu$ be a~probability measure such that
\begin{equation}\label{3.4abl1}
m_k=m_k(\mu):=\int\limits_{\Bbb R}u^k\,\mu(du)<\infty,\qquad k=1,\dots,2n.
\end{equation}
Then the~following relation holds
\begin{equation}\label{3.4abl2}
\lim_{z\to\infty}z^{2n+1}\Big(G_{\mu}(z)-\frac 1z-\frac{m_1}{z^2}-
\dots-\frac{m_{2n-1}}{z^{2n}}\Big)=m_{2n}
\end{equation}
uniformly in the~angle $\delta\le\arg z\le\pi-\delta$, 
where $0<\delta<\pi/2$.

Conversely, if for some function $G(z)\in\mathcal N$ the~relation $(\ref{3.4abl2})$
holds with real numbers $m_k$ for $z=iy,y\to\infty$, then $G(z)$ admits
the~representation~$(\ref{2.1*})$, where $\mu$ is a probability measure with moments $(\ref{3.4abl1})$.
\end{lemma}

By Lemma~\ref{3.4abl} and the~Cartier--Good formula for free random variables 
(see Lehner~\cite{Leh:2004}),
we easily obtain the~expansion of the~function $\phi_{\mu}(z)$. See \cite{Ka:2007} as well.

\begin{proposition}\label{3.5pr}
For every probability measure $\mu$ such that $m_{2n}(\mu)<\infty$ for a~non-negative
integer $n$ we have
\begin{equation}\label{3.5pro1}
\phi_{\mu}(z)=\kappa_1+\frac{\kappa_2}z+\dots+\frac{\kappa_{2n}}{z^{2n-1}}
+\frac{o(1)}{z^{2n-1}},\qquad z\in\Gamma_{\alpha,\beta},\quad z\to\infty,
\end{equation}
where $\kappa_j=\kappa_j(\mu),\,j=2,\dots,2n$, are the~free cumulants of the~probability measure $\mu$. 

Conversely, if for some function $\phi_{\mu}(z)$ the~relation $(\ref{3.5pro1})$
holds with real coefficients $\kappa_j$, then $\mu$ has a finite~moment $m_{2n}(\mu)<\infty$ and
$\kappa_j=\kappa_j(\mu),j=1,\dots,2n$.
\end{proposition}

We also need the~following well-known result (see for example~\cite{NiSp:2006}).
\begin{proposition}\label{3.4abcpr}
In order that a~probability measure $\mu$ has a~compact support it is necessary and sufficient
that the~sequence $\{\kappa_s(\mu)\}_{s=1}^{\infty}$ of free cumulants of this measure
satisfies the~inequality
\begin{equation}
|\kappa_s(\mu)|\le c^s,\quad s\in\mathbb N,\notag
\end{equation}
with some constant $c>0$.
\end{proposition}

We introduced the~definition of the~free cumulants in Section~2. Let us recall the~definition
of mixed free cumulants as well. Let $T_1,\dots,T_n\in\tilde A_{as}$ be random variables
and let $T_j^{(k)},\,k=1,\dots,n$, denote free copies. Set $T_j^{\omega}$ as in
(\ref{2.2}). Following Lehner~\cite{Leh:2004}, the~nth mixed cumulant may be defined via
\begin{equation}\notag
\kappa_n(T_1,\dots,T_n)=\frac 1n\tau(T_1^{\omega}T_2^{\omega}\dots T_n^{\omega})
\end{equation}
in a~short way.
We will use the~following known results, see~~\cite{Leh:2004} and \cite{NiSp:2006}.
\begin{theorem}\label{3c.1th}
Consider a~non-commutative probability space $(A,\tau)$ and let $(\kappa_n)_{n\in\mathbb N}$
be the~corresponding free cumulant functionals. Consider random variables $(T_j)_{j\in I}$
in $\tilde{A}_{sa}$. Then the~following two statements are equivalent.

\quad(i) $(T_j)_{j\in I}$ are freely independent.

\quad(ii) We have $\kappa_n(T_{j_1},\dots,T_{j_n})=0$ for all $n\ge 2$ 
and $j_1,\dots,j_n\in I$
such that at least two of these $n$ indices $j_h$ are different.
\end{theorem}
\begin{proposition}\label{3c.1pr}
Let $L_k=\sum_{j=1}^n a_{kj}T_j,\,k=1,\dots,m$, be an~affine transformation of $T_1,\dots,T_n$,
then we have, for $m\ge 2$,
\begin{equation}\notag
\kappa_m(L_1,\dots,L_m)=\sum_{j_1,\dots,j_m}a_{1,j_1}\dots a_{m,j_m}\kappa_m(T_{j_1},\dots,T_{j_m}).
\end{equation} 
\end{proposition}

\begin{proposition}\label{3c.2pr}
Mixed cumulants vanish. That is, if there is a~nontrivial subset $I\subset[n]$
$($i.e., $I\ne\emptyset$ and $I\ne[n]$, $[n]:=\{1,\dots,n\})$ such that $\{T_j\}_{j\in I}$ and 
$\{T_j\}_{j\in [n]\setminus I}$ are free, then $\kappa_n(T_1,\dots,T_n)=0$.
\end{proposition}

Now we prove an~analogue of a~known lemma of Linnik
for the~characteristic functions (see~\cite{ILin:1975}, Ch. 1, \S 6).

Recall that a~probability measure $\mu$ is symmetric if $\mu(S)=\mu(-S)$
for any real Borel set $S$. It is not difficult to verify 
(see~\cite{ChG:2006}) that $\mu$ is symmetric if and only if
$\phi_{\mu}(iy)$ takes imaginary values for $y>0$, where $\phi_{\mu}(iy)$
is defined.

\begin{lemma}\label{3.5l}
Let $\mu$ be a~symmetric probability measure and $\{y_k\}$ be a~sequence of
positive numbers such that $\lim y_k\to\infty$. If, for all $k$, 
$\phi_{\mu}(iy_k)=\phi_{\nu}(iy_k)$, where $\nu$ is a~symmetric probability measure with compact
support, then $\mu=\nu$.
\end{lemma}
\begin{proof}
We shall show that $\mu$ has moments $m_n(\mu):=\int_{\Bbb R}u^n\,\mu(du)$ 
of all orders and that $m_n(\mu)=m_n(\nu)$ for all $n=1,\dots$. 
The~proof proceeds by induction for even $n$. 

From the~assumptions of the~lemma we see  
\begin{equation}\label{3.5l1}
G_{\mu}(it_k)=G_{\nu}(it_k),\quad k\ge k_0,
\end{equation} 
where $t_k:=-iF_{\nu}(iy_k)\to\infty$ as $y_k\to\infty$, and $k_0$ is sufficiently 
large positive integer.
By (\ref{3.5l1}), we obtain the~following equation, using the~symmetry 
of the~measures $\mu$ and $\nu$,
\begin{equation}\label{3.5l2}
(it_k)^3\Big(G_{\mu}(it_k)-\frac 1{it_k}\Big)=
\int\limits_{\Bbb R}\frac{t_k^2u^2}{u^2+t_k^2}\,\mu(du)
=\int\limits_{\Bbb R}\frac{t_k^2u^2}{u^2+t_k^2}\,\nu(du).
\end{equation} 

We shall prove by induction that $m_{n}(\mu)=m_{n}(\nu)$ for all $n=0,1,\dots$.
Letting $t_k\to\infty$, we conclude from (\ref{3.5l2}) that $m_2(\mu)<\infty$ and $m_2(\mu)=m_2(\nu)$.
Now suppose that, for all $p<n$, $m_{2p}(\mu)$ exists and $m_{2p}(\mu)
=m_{2p}(\nu)$. Using (\ref{3.5l1}) and the~formula
$$
(it_k)^{2n+1}\Big(G_{\mu}(it_k)-\frac 1{it_k}-\frac{m_2(\mu)}{(it_k)^3}
-\dots-\frac{m_{2n-2}(\mu)}{(it_k)^{2n-1}}\Big)=
\int\limits_{\Bbb R}\frac{t_k^2u^{2n}}{u^2+t_k^2}\,\mu(du)
$$
we arrive at the~relation
$$
\int\limits_{\Bbb R}\frac{t_k^2u^{2n}}{u^2+t_k^2}\,\mu(du)=
\int\limits_{\Bbb R}\frac{t_k^2u^{2n}}{u^2+t_k^2}\,\nu(du).
$$
Letting here $t_k\to\infty$, we obtain $m_{2n}(\mu)<\infty$ and 
$m_{2n}(\mu)=m_{2n}(\nu)$ that was to be proved.

It remains to note that since $m_{n}(\mu)=m_{n}(\nu)$ for all $n=0,1,\dots$
and the~measure $\nu$ has compact support, we have $\mu=\nu$. 
Hence the~lemma is proved.
\end{proof}

Voiculescu~\cite{Vo:1992}, Maassen~\cite{Ma:1992} and in the~general case
Biane~\cite{Bi:1998} proved that  
there exist unique functions $Z_1(z)$ and $Z_1(z)$ from the~class
$\mathcal F$ such that
\begin{equation}\label{3a.1a}
z=Z_1(z)+Z_2(z)-F_{\mu_1}(Z_1(z))\quad\text{and}\quad
F_{\mu_1}(Z_1(z))=F_{\mu_2}(Z_1(z)),\quad z\in\Bbb C^+.
\end{equation}
In addition $F_{\mu_1\boxplus\mu_2}(z)=F_{\mu_1}(Z_1(z))$.
The~relation (\ref{3a.1a}) was proved by purely analytic methods
by Chistyakov and G\"otze~\cite{ChG:2005} and Belinschi and Bercovici~\cite{BelBe:2007}.

Introduce the~class ${\mathcal K}[a,b]$ in the~following way.
A~function $F(z)$ is in class ${\mathcal K}[a,b]$ if

1) $F(z)$ is in class $\mathcal N$, and

2) $F(z)$ is holomorphic and positive in the~interval $(-\infty,a)$,
and holomorphic and negative in the~interval $(b,+\infty)$.
The~following theorem is due to Krein~\cite{KrN:1977}.

\begin{theorem}\label{3.5th}
A~function $F(z)$ is in class ${\mathcal K}[a,b]$ if and only if it
admits a~representation
$$
F(z)=\int\limits_{[a,b]}\frac{\sigma(dt)}{t-z},
$$
where $\sigma$ is a~finite non-negative measure.
\end{theorem}

We now prove a free analogue of result by Wintner, see~\cite{LiO:1977}, Ch.~3, \S 2.

\begin{lemma}\label{3.6lem}
Assume that $\mu=\mu_1\boxplus\mu_2$, where $\mu$ has compact support.
Then $\mu_1$ and $\mu_2$ have compact support as well.
\end{lemma}
\begin{proof}
It suffices to prove the~lemma for the~measure $\mu_1$. 
The~proof for the~measure $\mu_2$ is similar.
By (\ref{3a.1a}), there exists $Z(z)\in\mathcal F$ such that
$F_{\mu}(z)=F_{\mu_1}(Z(z)),\,z\in\Bbb C^+$. Hence we obtain the~relation
\begin{equation}\label{3.6lem1}
\int\limits_{[-d,d]}\frac{\mu(du)}{z-u}=
\int\limits_{\Bbb R}\frac{\mu_1(du)}{Z(z)-u},\quad z\in\Bbb C^+,
\end{equation}
with $0<d<\infty$. Since $z/(Z(z)-u)\to 1$ as $z\to\infty$ non-tangentially to $\mathbb R$
and $\Im (Z(z)-u)\ge 0$ for $z\in\Bbb C^+$, then $Z(z)-u\in\mathcal F$ and we may
write
\begin{equation}\label{3.6lem2}
\frac 1{u-Z(z)}=\int\limits_{\Bbb R}\frac{\sigma(u,ds)}{s-z},\quad
z\in\Bbb C^+,
\end{equation}
where $\sigma(u,ds)$ is a~probability measure for every $u\in\Bbb R$ and $\sigma(u,S)$ is 
a~measurable function for every Borel $S$ set in $\Bbb R$. Using
this representation we deduce from (\ref{3.6lem1})
\begin{equation}\label{3.6lem3}
\mu(S)=\int\limits_{\Bbb R}\sigma(u,S)\,\mu_1(du)
\end{equation}
for every Borel $S$ set in $\Bbb R$. Let $S_0:=(-\infty,-d)\cup(d,\infty)$. We see from (\ref{3.6lem3}) 
that $\mu_1(A)=1$, where $A:=\{u\in\mathbb R:\sigma(u,S_0)=0\}$. 
Therefore, 
for every point $u_0\in A$, the~measure $\sigma(u_0,ds)$ has support contained in $[-d,d]$.    

It remains to show that $\mu_1$ has a~bounded support. Return to
(\ref{3.6lem2}) with $u=u_0\in A$. By Theorem~\ref{3.5th}, the~function
$Z(z)-u_0$ is holomorphic and real for $z=x<-d$ and for $z=x>d$. Since
$Z(z)$ admits the representation 
$$
Z(z)=\alpha+z+\int\limits_{\Bbb R}\Big(\frac 1{t-z}-\frac t{1+t^2}\Big)
(1+t^2)\,\nu(dt),
$$ 
where $\alpha\in\Bbb R$ and $\nu$ is a~finite nonnegative measure,
it follows from the~Stieltjes--Perron inversion formula that the~measure $\nu$ 
has bounded support contained in $[-d,d]$. Thus
\begin{equation}\label{3.6lem5}
Z(z)-u_0=\gamma-u_0+z+\int\limits_{[-d,d]}\frac{(1+t^2)\,\nu(dt)}{t-z},
\end{equation}
where $\gamma\in\Bbb R$. The~parameter $\gamma$ and the~measure
$\nu$ depend on $Z$ only and do not depend on $u_0$. Let $u_0>0$ and
be sufficiently large, i.e., $u_0>c(Z)>0$. Then, by (\ref{3.6lem5}),
$Z(x)-u_0<0$ for $x=u_0/2>2d$, a~contradiction with (\ref{3.6lem2}) 
for $u=u_0$ and $x=u_0/2$. An~analogous argument holds for $u_0<0$. Hence there exists
$c(Z)>0$ such that the~points $u_0\in A$ 
satisfy the~inequality $|u_0|\le c(Z)$. The~lemma is proved.
\end{proof}

\section{Auxiliary results on special functional equations.}

In this section we first describe some results (see Kagan, Linnik, Rao~\cite{KLR:1973})
on continuous solutions of special equations which were used to
characterize distributions via 
independence of linear statistics and  
identical distribution of linear statistics. 
\begin{lemma}\label{3.1aal}
Consider the~following equation, for $|u|<\delta_0,\,|v|<\delta_0$,
\begin{equation}
\psi_1(u+b_1v)+\dots+\psi_r(u+b_rv)=A(u)+B(v)+P_k(u,v),\notag
\end{equation}
where $P_k$ is a~polynomial of degree $k$; $\psi_j,\,A$ and $B$ are complex valued
functions of two real variables $u$ and $v$. We assume that

$(i)$  
the~numbers $b_j$ are all distinct 

$(ii)$ the~functions $A,\,B$, and $\psi_j$ are continuous. 

Then, in some neighborhood of the~origin, all the~functions $A,\,B$, and $\psi_j$ 
are polynomials of degree at most $\le\max(r,k)$. 
\end{lemma}
Consider the~equation
\begin{equation}\label{3.1}
\int\limits_0^1v(st)\,dQ_1(s)=\int\limits_0^1v(st)\,dQ_2(s),\quad\text{for all}\quad
0<t<1,
\end{equation}
for a~bounded continuous function $v(t)$ defined on $(0,1)$. Here $Q_1(s)$
and $Q_2(s)$ are nondecreasing functions satisfying the~condition
\begin{equation}\label{3.2}
\int\limits_0^1s^{-b}\,d(Q_1(s)+Q_2(s))<\infty
\end{equation}
for some $b>0$. We assume that the relation (\ref{3.1}) is nondegenerate, 
i.e., $Q_2(s)-Q_1(s)\not\equiv const$.

Applying a~Mellin transform to (\ref{3.1}) we easily obtain, for $Q(s):=Q_1(s)-Q_2(s)$,
\begin{equation}\label{3.3}
\int\limits_0^1t^{z-1}\,dt\int\limits_0^1v(st)\,dQ(s)
=\int\limits_0^1s^{-z}\,dQ(s)\int\limits_0^st^{z-1}v(t)\,dt=0
\quad \text{for all}\quad 0<\Re z<b.
\end{equation}
In view of (\ref{3.2}) we deduce from 
(\ref{3.3}), for $0<\Re z<b$,
\begin{equation}\label{3.4}
\Lambda(z)X(z;v)-K(z;v)=0,
\end{equation}
where
\begin{equation}\label{3.5}
\begin{split}
\Lambda(z):=\int\limits_0^1s^{-z}\,dQ(s),\qquad
X(z;v):=\int\limits_0^1t^{z-1}v(t)\,dt,\\
K(z;v):=\int\limits_0^1s^{-z}\,dQ(s)\int\limits_s^1t^{z-1}v(t)\,dt.
\end{split}
\end{equation}
The~functions $\Lambda(z)$ and $K(z;v)$ are analytic in the~half-plane
$\Re z<b$, and the~function $X(z;v)$ is analytic in the~half-plane
$\Re z>0$. We use the~relation (\ref{3.4}) for analytic continuation of
$X(z;v)$ into the~half-plane $\Re z\le 0$ as a~meromorphic function.
Keeping the~same notation, we have
\begin{equation}\label{3.6}
X(z;v)=K(z;v)/\Lambda(z),\qquad \Re z<b.
\end{equation}
The~singularities of $X(z;v)$ in the~half-plane $\Re z\le 0$ happen 
to be poles distributed among the~zeros of $\Lambda(z)$.

Taking an~arbitrary $\lambda>0$,
the~inversion formula for the~Mellin transform yields 
\begin{equation}\label{3.6a}
\int\limits_0^t u^{\lambda-1}\log(t/u)v(u)\,du=\frac 1{2\pi i}
\int\limits_{x-i\infty}^{x+i\infty}t^{\lambda-z}\frac{K(z;v)}{(z-\lambda)^2
\Lambda(z)}\,dz,\,\,0<x<\min(b,\lambda).
\end{equation}

Let $z_0$ be some zero of $\Lambda(z)$ of multiplicity $m_0$ in 
the~half-plane $\Re z\le 0$. We see that
\begin{equation}\label{3.7}
Res\Big(t^{\lambda-z}\frac{K(z;v)}{(z-\lambda)^2\Lambda(z)}\Big)=P_{z_0}(\log t)
t^{\lambda-z_0},\quad 0<t<1,
\end{equation}
where $P_{z_0}(t)$ is a~polynomial of degree at most $m_0-1$.
These residues may depend on $\lambda$, but it easily seen that
the~degree of the~polynomial $P_{z_0}(t)$ does not depend on $\lambda$.

If $P_{z_0}(t)\not\equiv 0$, we call the~number $-z_0$ 
an~{\it active exponent} of the~solution $v(t)$ and the~number $deg P_{z_0}(t)+1$ 
will be called 
the~{\it multiplicity} of the~active exponent $\xi=-z_0$. The~leading
coefficient of $P_{z_0}(t)$ will be denoted $a_{-z_0}(v)$ and the~degree
by $m_{-z_0}(v)$. All the~active exponents of $v(t)$ are located in 
the~half-plane $\Re z\ge 0$. 

If the~number of the~active exponents $\{z_k\}_{k=1}^d$ of $v(t)$ 
is finite, then, by Jordan's lemma on residues, it follows from (\ref{3.6a})
that
\begin{equation}\label{3.7a}
\int\limits_0^t u^{\lambda-1}\log(t/u)v(u)\,du=t^{\lambda}
\sum_{k=1}^d P_{z_k}(\log t)t^{-z_k},\quad 0<t<1.
\end{equation}

We shall introduce the~notation
\begin{equation}\label{3.8}
\sigma_1(v):=\inf \{\Re \xi,\,\,\,\xi\text{ active exponents of } v(t)\}.
\end{equation}

We need the~following results on active exponents and differentiable
solutions $v(t)$.
\begin{lemma}\label{3.1lem}
If $v(t)\ge 0$ is a~continuous solution of $(\ref{3.1})$ such that
$v(t)\to 0$ as to $0+$, then $\sigma_1(v)$ is an~active exponent and
\begin{equation}\label{3.9}
\sigma_1(v)>0\qquad\text{and}\qquad a_{\sigma_1}(v)>0.
\end{equation}
\end{lemma}

\begin{lemma}\label{3.3lem}
If under the~hypothesis of Lemma~$\ref{3.1lem}$ the~function $v(t)$ has 
a~continuous derivative $v^{(n)}(t)$ for some $n\ge 1$ and for all $0<t<1$
and the~following limit exists and is finite
$$
\lim_{t\to 0+}v^{(n)}(t)=v^{(n)}_+(0),
$$
then all the~active exponents $\xi$ of $v(t)$ which are not simultaneously integers 
and simple active exponents satisfy the~condition $\Re \xi>n$.
\end{lemma}

\section{Proof of an~analogue of the Darmois--Skitovich theorem.}

In this section we prove Theorem~\ref{2.1ath}.

In order to prove Theorem~\ref{2.1ath} we follow the~proof of 
the~classical Darmois-Skitovich theorem (see \cite{KLR:1973}). 
\begin{proof}{\it Necessity}.
Assume that free random variables $T_1,\dots,T_n$ with distributions 
$\mu_1,\dots$, $\mu_n$, respectively, are such that the~linear statistics $L_1$ and $L_2$
(see (\ref{2.1})) are free and the~coefficients $a_j,b_j$ of these  statistics 
satisfy the~assumptions of Theorem~\ref{2.1ath}.
Then for every pair of real numbers $(u,v)$ the~linear statistics $u L_1$ and
$v L_2$ are free and we have the~relation
\begin{align}
L:=u L_1+v L_2&=(u a_1+v b_1)T_1+\dots+(u a_m+v b_m)T_m\notag\\
&+(u a_{m+1}+v b_{m+1})T_{m+1}+\dots+(u a_n+v b_n)T_n.\label{DS.1}
\end{align}
Using (\ref{2.3*}), we deduce from (\ref{DS.1}) that 
\begin{align}
\phi_{(u a_1+v b_1)T_1}(z)&+\dots
+\phi_{(u a_n+v b_n)T_m}(z)\notag\\
&=\phi_{u L_1}(z)+\phi_{v L_2}(z)-\phi_{(u a_{m+1}+v b_{m+1})T_{m+1}}(z)-\dots
-\phi_{(u a_n+v b_n)T_n}(z)
\label{DS.2}
\end{align}
for $z\in\Gamma_{\alpha,\beta}$ with some $\alpha>0$ and $\beta>0$, 
where all functions $\phi_{(u a_j+v b_j)T_j}(z),j=1,\dots,n$,
and $\phi_{u L_1}(z),\,\phi_{v L_2}(z)$ are defined. Hence (\ref{DS.2}) holds
for $z=i$ and for $|u|\le\delta$ and $|v|\le\delta$ with sufficiently small $\delta>0$.
Note that the~functions $\phi_{(u a_{m+1}+v b_{m+1})T_{m+1}}(z),\dots,$ $\phi_{(u a_n+v b_n)T_n}(z)$ 
depend on $u$ or $v$ only.
Consider the~functions $\psi_j(w):=w\phi_{T_j}(i/w),\,j=1,\dots,n$, for $w\in\mathbb R$
and $|w|\le\delta'$ with sufficiently small $\delta'>0$. Since $\phi_{wT_j}(i)=w\phi_{T_j}(i/w)$, 
and $w\phi_{T_j}(i/w)\to 0$ as $w\to 0$, we see that (\ref{DS.2}) with $z=i$ has the~form 
\begin{equation}\label{DS3}
\psi_1(ua_1+vb_1)+\dots+\psi_m(ua_m+vb_m)=A(u)+B(v),\quad |u|<\delta, \,\,\,|v|<\delta,
\end{equation}
where $\psi_j,\,j=1,\dots,m$, and $A,B$ are complex-valued continuous functions, and $b_j/a_j,
\,j=1,\dots,m$ are all distinct. Then, by Lemma~\ref{3.1aal}, the~functions 
$\psi_j,\,j=1,\dots,m$, are polynomials of degree $\le m$. Therefore we have 
the~representation
\begin{equation}\label{DS4}
\phi_{T_j}(z)=z\sum_{s=0}^m\frac {d_{sj}}{z^{s}},\quad j=1,\dots,m,
\end{equation}
for $z\in\Gamma_{\alpha',\beta'}$ with some $\alpha'>0$ and $\beta'>0$, where $d_{sj}$
are complex valued coefficients. Since $\phi_{T_j}(iy)=o(y)$ as $y\to\infty$, all $d_{0j}=0$.
In view of the~relations
\begin{equation}\notag
\phi_{-T_j}(iy)=-\Re\phi_{T_j}(iy)+i\Im\phi_{T_j}(iy),\quad y\ge \beta',
\end{equation}
we see that
\begin{equation}\notag
2\Im\phi_{T_j}(iy)=\phi_{T_j}(iy)+\phi_{-T_j}(iy)=\sum_{l=1}^{[m/2]}(-1)^l\frac {d_{2l,j}}{y^{2l-1}}
\end{equation}
and
\begin{equation}\notag
2\Re\phi_{T_j}(iy)=\phi_{T_j}(iy)-\phi_{-T_j}(iy)=\sum_{l=0}^{[(m-1)/2]}(-1)^l\frac {d_{2l+1,j}}{y^{2l}}
\end{equation}
for $y\ge \beta'$. We easily deduce from the last two relations that the~coefficients $d_{sj}$ are real-valued. 
Then, by Proposition~\ref{3.5pr}, 
all moments $m_k(\mu_j)$ exist
and $d_{sj}=\kappa_{s+1}(\mu_j),\,s=1,\dots,m$ and $\kappa_{s}(\mu_j)=0$ for $s>m$.
By Proposition~\ref{3.4abcpr}, the~measures $\mu_j,\,j=1,\dots,m$, have compact supports.
We now return to the~relation~(\ref{DS3}). By (\ref{DS4}), the~functions on both sides
of (\ref{DS3}) are differentiable. Differentiating sequentially both sides of (\ref{DS3}) 
with respect to $u$ and $v$ we obtain a~relation from which (\ref{2.4}) follows immediately.

{\it Sufficiency}. Assume that free random variables $T_1,\dots,T_n$ satisfy the~assumptions of 
Theorem~\ref{2.1ath}. Consider mixed cumulants $\kappa_s(L_{j_1},L_{j_2},\dots,L_{j_s})$
such that $q$ indices $j_h$ are equal 1 and $s-q$ indices $j_h$ are equal 2. 
Then, by Propositions~\ref{3c.1pr},~\ref{3c.2pr}, and by (\ref{2.4}), 
$L_1$ and $L_2$ have vanishing mixed cumulants 
\begin{equation}\notag
\kappa_s(L_{j_1},L_{j_2},\dots,L_{j_s})
=\sum_{j=1}^ma_j^qb_j^{s-q}\kappa_s(\underbrace{T_j,\dots,T_j}_{s\rm\;times})=
\sum_{j=1}^ma_j^qb_j^{s-q}\kappa_s(T_j)=0
\end{equation}
for $s=2,\dots,m$ and $q=1,\dots,s-1$. In addition we clearly have $\kappa_s(L_{j_1},L_{j_2},
\dots,L_{j_s})=0$ for all $s\ge m+1$ and $q=1,\dots,s-1$. 
Hence, by Theorem~\ref{3c.1th}, the~linear forms $L_1$ and $L_2$ are free independent. 

The~theorem is completely proved.
\end{proof}

\section{Characterization of free cumulants}

First we shall prove Theorem~\ref{2.2ath}.
\begin{proof}
{\it Sufficiency}. 
We conclude from the definition of the domain $\Omega_{\varphi}$ that its boundary is 
a Jordan curve consisting of a finite number of curves parametrized by algebraic functions. Moreover, by
the~assumptions of the~theorem, there exist real numbers $a<0$ and $b>0$ such that 
$\gamma:=\gamma_1\cup\gamma_2\cup\gamma_3$, where $\gamma_1$ is the~half-line $z=x$ 
with $x\le a$, $\gamma_2$ is is a Jordan curve laying in $\mathbb C^+$ and connecting the point $a$ 
and $b$, $\gamma_3$ is the~half-line $z=x$ with $x\ge b$. 

Note that the~function $\varphi(z):\Omega_{\varphi}\to \Bbb C$ 
is analytic such that  
\begin{equation}\label{CH.C.2}
\lim_{R\to+\infty}\max_{|z|=R}|\varphi(z)|=0.
\end{equation}
We shall show that the~function $f:\Omega_{\varphi}\to\Bbb C$ 
defined via $z\mapsto z+\varphi(z)$ takes every value in $\mathbb C^+$ precisely once.
The~inverse $f^{(-1)}:\Bbb C^+\to \Bbb C^+$ thus defined is in the~class 
$\Cal F$.

Let $R$ be a~sufficiently large positive number.
For every fixed $w\in\Bbb C^+$
we consider a~closed rectifiable curve $\gamma_4$ 
consisting of a~curve $\gamma_{4,1}$, which is a~part of the~curve 
$\gamma$, connecting $-R$ to $R$, and
the~circular arc $\gamma_{4,2}:z=Re^{i\theta}$ with $0<\theta<\pi$ 
connecting $R$ to $-R$. The~curve $\gamma_4=\gamma_4(R)$ depends on $R$.
 
We see from the~construction of the~curve $\gamma_{4,1}$ that 
if $z$ runs  through $\gamma_{4,1}$ the~image $\zeta=f(z)$ lies on the~interval
$[-A_{-R},A_R]$, where $f(-R)=-A_{-R}$ and $f(R)=A_R$. Here
$A_{\pm R}\to\infty$ as $R\to\infty$. We note as well that if $z$ runs  
through $\gamma_{4,2}$ the~image $\zeta=f(z)$ lies in the~domain 
$|\zeta|\ge \min\{A_{-R},A_R\}/2,\,\Im \zeta>0$. 
 
Hence $f(z)$ winds around $w\in\mathbb C^+$ once when $z$ runs through $\gamma_4$, and it follows from the~argument
principle that there is a~unique point $z_0$ in the~interior of the~curve $\gamma_4$
such that $f(z_0)=w$. Since this relation holds for all
all curves $\gamma_{4}=\gamma_4(R)$ with sufficiently large $R>0$ , we deduce 
that the~point $z_0$ is unique in $\Omega_{\varphi}$.

Therefore the~inverse function $f^{(-1)}:\Bbb C^+\to\Omega_{\varphi}$ exists 
and is analytic on $\Bbb C^+$. By condition (\ref{CH.C.2}), $z/f^{(-1)}(z)\to 1$
as $z\to \infty$ non-tangentially to $\mathbb R$ and therefore $f^{(-1)}(z)\in\mathcal F$. 
Hence there exists a~probability measure $\mu$ such that $\varphi(z)=\phi_{\mu}(z)$ and
the~sufficiency of the~assumptions of the~theorem is proved.

{\it Necessity}. Let there exist a Jordan curve $\gamma$, laying in $\Omega_{\varphi}$ 
and connecting $0$ and $\infty$.
We shall show that in this case it does not exist
a~probability measure $\mu$ such that $\varphi(z)=\phi_{\mu}(z)$. Assume to the~contrary that
there exists a~probability measure $\mu$ such that $\varphi(z)=\phi_{\mu}(z)$ 
for $z\in\Gamma_{\alpha,\beta}$ with some positive $\alpha$ and $\beta$. Then
$\varphi(z)=\phi_{\mu}(z)$ for $z\in\mathbb C^+$
and $|z|\ge c$ with a~sufficiently large constant $c>0$. By Proposition~\ref{3.4abcpr},
$\mu$ has a compact support.  
It is obvious that the~relation
\begin{equation}\label{CH.C.3}
F_{\mu}(z+\varphi(z))=z
\end{equation}
holds for $z\in\mathbb C^+,\,|z|\ge c$ and therefore it holds for $z\in\Omega_{\varphi}$.  
Since
$\varphi(z)\to\infty$ as $z\to 0$, and $F_{\mu}(z)=(1+o(1))z$ as $z\to\infty$,  
the~relation (\ref{CH.C.3}) with $z\in\gamma$ and $z\to 0$ leads to a~contradiction. This proves the~necessity
of the~assumptions of the~theorem and completely proves the~theorem.
\end{proof}

{\bf Proof of Remark~\ref{2.2arem}}. Since probability measures corresponding to the sequences of free cumulants
of the set $S$ have the uniformly bounded support, the conclusion of the remark follows immediately
from the fact that $S$ is conditionally compact.
$\square$

{\bf Proof of Corollary~\ref{2.2acor}}.
Let $0<\kappa_4<\frac 14$. Without loss of generality we assume that $\kappa_3\le 0$. 
It follows from Theorem~\ref{2.2ath} that if the~sequence
$0,1,\kappa_3,\kappa_4,0,\dots$ is a~sequence of free cumulants of some probability measure
then the polynomial $P(r,x):=r^4-r^2-2\kappa_3xr+(1-4x^2)\kappa_4$ has at the least one positive root
for every fixed $x\in[-1,1]$. Denote $r_{\max}(x)$ the maximum of such roots.

Assume that $x\in[-1,-1/2]$. In this case $P(r,x),\,r>0$, has one positive root $r_{1,1}(x)=r_{\max}(x)$ 
and this root is a continuous function on $[-1,-1/2]$. Let $x\in(-1/2,0)$. Then $P(r,x)$
has two positive roots only, say $r_{1,2}(x)$ and $r_{2,2}(x)$ ($r_{1,2}(x)< r_{2,2}(x)=r_{\max}(x))$,
and $r_{1,2}(x)\to 0$ as $x\to -1/2$. These roots are continuous functions on $(-1/2,0)$ and 
$\lim_{x\to -1/2-0}r_{1,1}(x)=\lim_{x\to -1/2+0}r_{2,2}(x)$. 
In addition $P(r,x)>0$ for $r>r_{2,2}(x)$ and for $0<r<r_{1,2}(x)$.

Consider the function
\begin{equation}
\rho (x,\kappa_4):=r(x,\kappa_4)(1-2r^2(x,\kappa_4)),\quad\text{where}\quad r(x,\kappa_4):=\Big\{ 
\frac 12\Big(\frac 13+\sqrt{\frac 19-\frac 43(4x^2-1)\kappa_4}\Big)\Big\}^{1/2}\notag
\end{equation}
for $0\le x\le x_1=x_1(\kappa_4):=\min\{1,\frac 12\sqrt{1+\frac 1{12\kappa_4}}\}$.
Note that $x_1=1$ for $\kappa_4\le 1/36$ and $1/\sqrt 3<x_1<1$ for $1/36<\kappa_4<1/4$.

It is easy to verify that the function $\rho(x,\kappa_4)/x,\,0\le x\le x_1$, is strictly monotone
for $0<x<x_2$ and for $x_2\le x\le x_1$ and
has a unique minimum at the point $x_2=x_2(\kappa_4)
:=\min\{x_1,\frac 12\sqrt{\frac 1{\sqrt{\kappa_4}}-2}\}$. 
Note that $x_2=\frac 12\sqrt{\frac 1{\sqrt{\kappa_4}}-2}$ if $\kappa_4\ge 1/36$. 
In this case $r(x_2,\kappa_4)=\root 4\of{\kappa_4}$.
Note as well that $x_2=1$ for $\kappa_4\le 1/36$ and $x_2<x_1<1$ for $1/36<\kappa_4<1/4$.
 
Now let us show that if $0,1,\kappa_3,\kappa_4,0,0,\dots$
is a free cumulant sequence, then 
\begin{equation}\label{2.2acor.1a}
-\kappa_3\le \rho(x_2,\kappa_4)/x_2.
\end{equation}

Let to the contrary $-\kappa_3=\rho(x^*,\kappa_4)/x^*$ for some $x^*\in(0,x_2)$. 
Then we have the inequality 
\begin{equation}\label{2.2acor.1}
-\kappa_3<\rho(x,\kappa_4)/x\quad\text{ for }\quad 0<x<x^*\quad\text{and}\quad
-\kappa_3>\rho(x,\kappa_4)/x\quad\text{ for }\quad x^*<x<x_2.
\end{equation}

Fix a~parameter $b:=-(4x^2-1)\kappa_4$ with $x\in(0,x_1)$. The~line $y=ar+b$ is a~tangent 
to the~curve $y=r^2(1-r^2)$ at the point $r=r(x,\kappa_4)$ iff $a=2\rho(x,\kappa_4)$.
 
Let $x\in(0,1/2)$.  
Then $P(r,x)$ has a positive root
iff $-\kappa_3x\le \rho(x)$. Assume that $x^*\in(0,1/2)$ then,
by (\ref{2.2acor.1}), this inequality holds for $x\in(0,x^*)$ and
there exists $x\in(x^*,1/2)$ such that $P(r,x)>0$ for $r>0$, a contradiction 
with the assumptions of Theorem~\ref{2.2ath}. Hence $x^*\notin(0,1/2)$.

It is easy to see that, for every fixed $x\in(0,1/2)$, $P(r,x)$ has a positive
root iff $P(r,x)$ has two positive roots only, say $r_{1,3}(x)$ and $r_{2,3}(x)$ 
($r_{1,3}(x)\le r_{2,3}(x)=r_{\max}(x))$, and $r_{j,3}(x),\,j=1,2$, are continuous functions on $(0,1/2]$. 
In addition $\lim_{x\to -0}r_{1,2}(x)=\lim_{x\to +0}r_{2,3}(x)$ and
$P(r,x)>0$ for $r>r_{2,3}(x)$ and for $0<r<r_{1,3}(x)$. Hence we can conclude that $r_{\max}(x)$ 
is a continuous function for $[-1,1/2]$. 

Now assume that $x^*\in [1/2,x_2)$.
Let $x\in[1/2,x^*)$. It is easy to see in this case, in view of (\ref{2.2acor.1}), that
$P(r,x)$ has  
three positive roots, 
say $r_{1,4}(x)< r_{2,4}(x)<r_{3,4}(x)=r_{\max}(x)$. These functions are continuous for $x\in[1/2,x^*]$
and $\lim_{x\to 1/2+0}r_{2,4}(x)=\lim_{x\to 1/2-0}r_{1,3}(x)$, 
$\lim_{x\to 1/2+0}r_{3,4}(x)=\lim_{x\to 1/2-0}r_{2,3}(x)$. In addition $r_{\max}(x)$
is a continuous function for $[-1,x^*]$. By (\ref{2.2acor.1}), we see that, for $x\in(x^*,x_2)$,
$P(r,x)$ has only one positive root $r_{1,4}(x)=r_{\max}(x)$.
Since three positive roots of $P(r,x),\,r>0$, with fixed $x\in(0,1)$ coincide at the point $x=x_1$
only, we note that $\lim_{x\to x^*-0}r_{2,4}(x)=\lim_{x\to x^*-0}r_{3,4}(x)$ and 
$\lim_{x\to x^*+0}r_{\max}(x)<\lim_{x\to x^*-0}r_{\max}(x)$.
Then there exists a Jordan curve 
in $\mathbb C^+$ containing $0$ and $\infty$ on which $\Im (z+\varphi(z))>0$, a contradiction.

Thus $x^*=x_2$ and (\ref{2.2acor.1a}) is proved. We can rewrite this assumption in the following form.

If $0<\kappa_4\le 1/36$, then $x_1=x_2=1$ and we have the estimate $-\kappa_3\le \rho(1,\kappa_4)$.

If $1/36<\kappa_4<1/4$, then $x_1<1$ and $x_2\le x_1$. In this case we have the estimate 
$-\kappa_3\le \rho(x_2,\kappa_4)/x_2$.

Now we assume that (\ref{2.2acor.1a}) holds. Let us show that $0,1,\kappa_3,\kappa_4,0,0,\dots$
is a free cumulant sequence.

Repeating the previous arguments we conclude that $r_{\max}(x)$ is a continuous function for $x\in [-1,x_1]$.

Finally note that if $x\in(x_1,1]$ and $x_1<1$, then, for such fixed $x$ the polynomial $P(r,x)$ 
has one positive root $r_{1,5}(x)$ only. The function $r_{1,5}(x)$ is continuous on 
$x\in(x_1,1]$ and such that $\lim_{x\to x_1-0}r_{3,4}(x)=\lim_{x\to x_1+0}r_{1,5}(x)$ and  
$r_{1,5}(1)>0$. Therefore $r_{\max}(x)$ is a continuous function for all $x\in[-1,1]$ and 
$r_{\max}(\pm 1)>0$.

Hence the assumption (\ref{2.2acor.1a}) is necessary and sufficient in order  
that the~sequence $0,1,\kappa_3,\kappa_4,0,\dots$ is a~sequence
of free cumulants of some probability measure in the case $0<\kappa_4<1/4$.

Assume that $\kappa_4>1/4$. Then $P(r,0)>0,\,r>0$, and, by Theorem~\ref{2.2ath}, the~sequence 
$0,1,\kappa_3,\kappa_4,0,\dots$ is not a~sequence of free cumulants of a probability measure.

We now assume that   
$-\frac 1{12}<\kappa_4<0$. Note that $P(r,1)$ has a positive root iff
$-\kappa_3\le \rho(1,\kappa_4)$. But it is easy to see that under this condition
$P(r,1)$ has a positive root for every $x\in[-1,1]$ and $r_{\max}(x)$ is a continuous function
on $[-1,1]$. Hence the condition $-\kappa_3\le \rho(1,\kappa_4)$ is necessary and sufficient in order
that the~sequence $0,1,\kappa_3,\kappa_4,0,\dots$ is a~sequence
of free cumulants of some probability measure.

Assume that $\kappa_4<-\frac 1{12}$ and $\kappa_3\in\mathbb R$. In this case 
$P(r,sign(\kappa_3))>0,\,r>0$, if $\kappa_3\ne 0$ and $P(r,1)>0,\,r>0$, if $\kappa_3=0$.

Assume that $\kappa_4=0$. We easily conclude from the previous arguments 
that the~condition $|\kappa_3|\le \frac 1{3\sqrt 3}$
is necessary and sufficient in order that the~sequence $0,1,\kappa_3,0,0,\dots$ is a~sequence
of free cumulants of some probability measure.

It remains finally to note that the assertion of the corollary for $\kappa_4=-1/12,1/4$ follows immediately
from Remark~\ref{2.2arem}.
$\square$

\section{ Necessity of conditions for the~characterization of semicircular measures. 
Auxiliary results}

In order to prove 
Theorem~2.1 we need the~following results.

The~first of them is a~description of $\boxplus$-stable distributions
(see \cite{BeP:1999} and \cite{BeVo:1993}).
\begin{lemma}\label{5.1lem}
Every $\boxplus$-stable probability measure is equivalent to a~unique probability measure
whose Voiculescu transform is given by one of the~following

$(1)$ $\phi(z)=z^{-1}$;

$(2)$ $\phi(z)=e^{i(\alpha-2)\rho\pi}z^{-\alpha+1}$ with $1<\alpha<2,
\,\,0\le\rho\le 1$;

$(3)$ $\quad (i)\,\,\,\phi(z)=0$,

$\,\,\quad\quad (ii)\,\,\,\phi(z)=-2\rho i+2(2\rho-1)/\pi\log z$ with $0\le\rho\le 1$;

$(4)$ $\phi(z)=-e^{i\alpha\rho\pi}z^{-\alpha+1}$ with $0<\alpha<1,
\,\,0\le\rho\le 1$.
\end{lemma} 

Here and in the~sequel we choose the~principal branch of the~functions $z^{-\alpha+1}$
and $\log z$.

The~stability index of a~$\boxplus$-stable probability measure is equal to $2$
in case (1), to $\alpha$ in cases (2) and (4), and to 1 in case (3).
The~parameter $\rho$ which appears in cases (2), (3) and (4) will be called 
the~{\it asymmetry coefficient}, and one can see that the~measure
corresponding to the~parameters $(\alpha,\rho)$ is the~image of the~measure
with parameters $(\alpha,1-\rho)$ by the~map $t\mapsto-t$ on $\Bbb R$.

The~next two lemmas are an~analog of a~result by Linnik (see \cite{Li:1963},
\cite{Li:1962}).
\begin{lemma}\label{5.2lem}
Let $\alpha>1$ and $\alpha\ne 2m+1$, where $m\in\mathbb N$.
The~function
$$
\phi(z)=\frac 1z-\varepsilon \cos\big(\alpha\pi/2\big)
\frac{ie^{i\alpha\pi/2}}{z^{\alpha}},
\quad z\in \Bbb C^+,
$$
with sufficiently small parameter $\varepsilon>0$ is the~Voiculescu
transform of some symmetric probability measure. 
\end{lemma}
\begin{proof}
Define the~following region:
$$
\Omega_{\alpha,\varepsilon}=\Big\{re^{i\theta}\in\Bbb C^+:0<\theta<\pi,\,\,
r^{\alpha}\Big(r-\frac 1r\Big)>b\frac{\cos\big(\alpha(\theta-\pi/2)\big)}
{\sin\theta}\Big\},
$$
where $b:=\varepsilon\cos(\alpha\pi/2)$. The~region 
$\Omega_{\alpha,\varepsilon}$ is a~Jordan domain with boundary 
curve $\gamma$ which is given by the~equation $z=r(\theta)e^{i\theta},\,0<\theta<\pi$, where $r(\theta)$ is defined by 
the~equation:
\begin{equation}\label{5.1}
r^{\alpha}\Big(r-\frac 1r\Big)=b\frac{\cos\big(\alpha(\theta-\pi/2)\big)}
{\sin \theta}.
\end{equation}
We see that (\ref{5.1}) has an~unique solution $r(\theta)$ which is greater than 1 
for $\theta$ such that 
$b\cos\big(\alpha(\theta-\pi/2)\big)>0$. If 
$b\cos\big(\alpha(\theta-\pi/2)\big)<0$, (\ref{5.1}) has   
two solutions $r(\theta)<1$. We choose the~larger of them. If 
$\cos\big(\alpha(\theta-\pi/2)\big)=0$, then $r=1$.

Note that the~function $\phi(z):\Omega_{\alpha,\varepsilon}\to \Bbb C$ 
is analytic with  
\begin{equation}\label{5.2}
\lim_{R\to+\infty}\max_{|z|=R,\,\Im z\ge 0}|\phi(z)|=0.
\end{equation}
We shall now show that the~function $f:\Omega_{\alpha,\varepsilon}\to\Bbb C$ 
defined via 
$z\mapsto z+\phi(z)$ takes every value in $\Bbb C^+$ precisely once.
The~inverse $f^{(-1)}:\Bbb C^+\to \Bbb C^+$ thus defined is in the~class 
$\Cal F$.

Denote by $a_R,\,\Re a_n>0$, a~point of an~intersection of the~curve $\gamma$ 
with the~circle $|z|=R$ with sufficiently large $R\ge R_0$.

For every fixed $w\in\mathbb C^+$
we consider a~closed rectifiable curve $\gamma_1$ 
consisting of some smooth curve $\gamma_{1,1}$, which is a~part of the~curve 
$\gamma$, connecting $-\overline{a}_R$ to $a_R$,
the~arc $\gamma_{1,2}$ of the~semicircle $z=Re^{i\theta},\,0<\theta<\pi$,
connecting $a_R$ to $-\overline{a}_R$. The~curve $\gamma_1$ depend on $R$.
 
We see from the~construction of the~curve $\gamma_{1,1}$ that 
if $z$ runs  through $\gamma_{1,1}$ the~image $\zeta=f(z)$ lies on the~interval
$[-A_R,A_R]$, where $-A_R=f(-\overline{a}_R)$ and $A_R=f(a_R)$. Here
$A_R\to\infty$ as $R\to\infty$. We note as well that if $z$ runs  
through $\gamma_{1,2}$ the~image $\zeta=f(z)$ lies in the~domain 
$|\zeta|\ge A_R/2,\,\Im \zeta>0$. 
 
Hence $f(z)$ winds around $w$ once when $z$ runs through $\gamma_1$, and it follows from the~argument
principle that inside the~curve $\gamma_1$ there is  a~unique 
point $z_0$ such that $f(z_0)=w$. Since this relation holds for all
sufficiently large $R>0$, we deduce 
that the~point $z_0$ is unique in $\Omega_{\alpha,\varepsilon}$.

Hence the~inverse function $f^{(-1)}:\Bbb C^+\to\Bbb C^+$ exists 
and is analytic in $\Bbb C^+$. By condition (\ref{5.2}), $z/f^{(-1)}(z)\to 1$
as $z\to\infty,\Im z>0$, non-tangentially to $\mathbb R$ and therefore $f^{(-1)}(z)\in\Cal F$.
This proves our~assertion.

\end{proof}

\begin{lemma}\label{5.3lem} 
The~function
$$
\phi(z)=\frac {1+\varepsilon\,(\log z-i\pi/2)}z,
\quad z\in \Bbb C^+,
$$
with sufficiently small parameter $\varepsilon>0$ is the~Voiculescu
transform of some symmetric probability measure.
\end{lemma}
\begin{proof}
Denote $z=re^{i\theta},\,r>0,0<\theta<\pi$, and consider the~function
\begin{equation}
\psi(r,\theta):=r\sin\theta+\Im \phi(r^{i\theta})=\Big(r-\frac 1r\Big)\sin\theta
-\frac{\varepsilon\sin\theta}r\log r+\frac{\varepsilon\cos\theta}r(\theta-\pi/2).
\notag
\end{equation}
We see from this formula that 
$\psi(1,\theta)\le 0$ for $0\le\theta\le\pi$. In addition, for every fixed $\theta\in (0,\pi)$,
$\psi(r,\theta)\to+\infty$ as $r\to+\infty$. Hence, 
for every fixed $\theta\in (0,\pi)$,
there exist points $r_j\ge 1$ such that $\psi(r_j,\theta)=0$.
Denote by $r(\theta)$ maximum of them. 
Note that $r(\theta)\to\infty$ as $\theta\to 0$ or $\theta\to\pi$.

Introduce the~curve $\gamma$ by the~equation $z=r(\theta)e^{i\theta},\,
0<\theta<\pi$. Denote by $\Omega_{\varepsilon}$ the~domain $\{z=r^{i\theta}:
r>r(\theta),\,0<\theta<\pi\}$.

Note that the~function $\phi(z):\Omega_{\varepsilon}\to \Bbb C$ 
is analytic with  
\begin{equation}\label{5.3}
\lim_{R\to+\infty}\max_{|z|=R,\,\Im z\ge 0}|\phi(z)|=0.
\end{equation}
We shall now show that the~function $f:\Omega_{\varepsilon}\to\Bbb C$ 
defined via 
$z\mapsto z+\phi(z)$ takes every value in $\mathbb C^+$ precisely once.
The~inverse $f^{(-1)}:\mathbb C^+\to \mathbb C^+$ thus defined is in the~class 
$\Cal F$.

We define a~closed rectifiable curve $\gamma_1$
in the~same way as in the~proof of Lemma~\ref{5.2lem}. Repeating the~argument
of the~proof of Lemma~\ref{5.2lem}, we conclude that if $z$ runs through 
$\gamma_1$ the~image $\zeta=f(z)$  
winds around every fixed point $w\in\mathbb C^+$ once, 
and it follows from the~argument
principle that inside the~curve $\gamma_1$ there is  a~unique 
point $z_0$ such that $f(z_0)=w$. Since this relation holds for all
sufficiently large $R>0$, we deduce 
that the~point $z_0$ is unique in $\Omega_{\varepsilon}$.

Hence the~inverse function $f^{(-1)}:\mathbb C^+\to\mathbb C^+$ exists 
and is analytic in $\mathbb C^+$. By condition (\ref{5.3}), $z/f^{(-1)}(z)\to 1$ 
as $z\to\infty$ non-tangentially to $\mathbb R$ and therefore $f^{(-1)}(z)\in\Cal F$.
This proves the~lemma. 
\end{proof}

\begin{remark}Using similar arguments as those in the~proof of Lemma~$\ref{5.3lem}$ 
one can prove a~more general result.

Let $m$ be a~positive integer. The~function
$$
\phi(z)=\frac {1+\varepsilon\,(\log z-i\pi/2)^m}z,
\quad z\in \Bbb C^+,
$$
with sufficiently small parameter $\varepsilon>0$ is the~Voiculescu
transform of some symmetric probability measure.
\end{remark}

\section{ Characterization of semicircular measures}

In this section we shall prove Theorem~\ref{2.1th}, \ref{2.1th*}, \ref{2.2th} and Corollary~\ref{2.1co}. 
We use in the~proof of these theorems some ideas of the~papers \cite{Li:1963}, 
\cite{Li:1962} and \cite{ZYa:1991}.

{\bf Proof of Theorem~$\ref{2.1th}$}. 

{\it Sufficiency}. Let $\mu$ be the~non-degenerate distribution of free random variables $T_1,\dots,T_n$
which satisfy the~relation $L_1\stackrel{D}{=} L_2$. The~Voiculescu transform $\phi_{\mu}(z)$
of the~probability measure $\mu$ is defined in a~domain 
$\Gamma_{\alpha,\beta}$ with some $\alpha>0$ and $\beta>0$. Since one can extend 
the~function $G_{\mu}(z)$
on $\mathbb C^-$ assuming $G_{\mu}(z)=\overline{G_{\mu}(\bar z)}$, we can extend
the~Voiculescu transform $\phi_{\mu}(z)$ on the~domain $-\Gamma_{\alpha,\beta}$
assuming $\phi_{\mu}(z)=\overline{\phi_{\mu}(\bar z)}$. Now we note that
the~Voiculescu transform $\phi_{\mu}(z)$ satisfies the~following equation
\begin{equation}\label{4.1}
\sum_{j=1}^n a_j\phi_{\mu}(z/a_j)=\sum_{k=1}^n b_k\phi_{\mu}(z/b_k)\quad\text{for all}
\quad z\in\Gamma_{\alpha,\beta}.
\end{equation}
Without loss of generality we assume that $\beta=1$. As shown in Section~3,
$\Im \phi_{\mu}(z)\le 0$ for $z\in\Gamma_{\alpha,1}$.
Denote
\begin{equation}\label{4.1*}
v(t):=t\Im \phi_{\mu}(i/t),\qquad 0<t<1.
\end{equation}
Note that the~function $v(t),t\in(0,1)$, is infinitely differentiable
and $v(t)\to 0$ as $t\to 0+$. Moreover $v(at)=v(-at)$ for real $a$ and 
$t\in(0,1)$, therefore it follows from (\ref{4.1}) that
\begin{equation}\label{4.2} 
\sum_{j=1}^n v(|a_j|t)=\sum_{k=1}^n v(|b_k|t),\qquad 0<t<1.
\end{equation}
In the~sequel we consider special solutions of these equations. 

We shall apply the~auxiliary results of Section~4 which describe
solutions of the~equation (\ref{4.2}) in the~case 
$Q(s)=Q_1(s)-Q_2(s)$, where $Q_1(s)$ and $Q_2(s)$ are distribution functions of
the~measures $\delta_{|a_1|}+\dots+\delta_{|a_n|}$ and $\delta_{|b_1|}+\dots+\delta_{|b_n|}$,
respectively, and 
$v(t):=t\Im \phi_{\mu}(i/t)$. First we shall prove the~following lemma.
\begin{lemma}\label{4.1lem}
The~parameter $\sigma_1(v)$, defined in $(\ref{3.8})$, for a~solution $v(t)$ of $(\ref{4.2})$
is an~active exponent and $0<\sigma_1(v)\le 2,\,a_{\sigma_1(v)}>0$.
\end{lemma}

This lemma shows that if $v(t)$ from (\ref{4.1*}) is a~solution of (\ref{4.2}),
then $\Lambda_1(z)$ has a~root $\gamma$ such that $0<\gamma\le 2$.

\begin{proof}
By Lemma~\ref{3.1lem}, we only need to prove the~inequality $\sigma_1(v)\le 2$. 
Let us assume to the~contrary that $\sigma_1(v)=2+\eta,\,\eta>0$. By the~definition
of $\sigma_1(v)$ we see that the~function $X(z;v)$ is analytic for
$\Re z>-2-\eta$. Since $v(t)\ge 0,\,0<t<1$, we conclude by L\'evy's and Raikov's theorem
(see \cite{LiO:1977}, Ch. 2, Theorem~2.2.1) that
$$
\int\limits_0^1t^{-3-\eta/2}v(t)\,dt<\infty.
$$
It follows from this relation that there exists a~sequence 
$\{t_l\}_{l=1}^{\infty}$ such that $t_l\to 0$ as $l\to\infty$ and
for which 
\begin{equation}\label{4.3a}
\lim_{l\to\infty}v(t_l)/t_l^2=0.
\end{equation}

By Proposition~\ref{3.4pr},
\begin{equation}\label{4.3}
\phi_{\mu}(z)=z^2\Big(G_{\mu}(z)-\frac 1z\Big)(1+q_{\mu}(z)),
\quad z\in\Gamma_{\alpha_1,\beta_1},
\end{equation}
where $|q_{\mu}(z)|=o(1)$ as $z\to\infty$ non-tangentially to $\mathbb R$. 
Denote by $\overline{\mu}$ the~probability measure such that
$\overline{\mu}(S):=\mu(-S)$ for any Borel set $S$. It is easy to see
that $\Im \phi_{\mu}(iy)=\frac 12\Im\phi_{\mu\boxplus\overline{\mu}}(iy)$
for $y\ge y_0>0$. In addition the~measure
$\mu\boxplus\overline{\mu}$ is symmetric. Therefore 
it easily follows from (\ref{4.3}) that the~relation
\begin{align}\label{4.4}
\Im \phi_{\mu}(iy)&=-\frac{y^2}2\Im\Big(G_{\mu\boxplus\overline{\mu}}(iy)
-\frac 1{iy}\Big)(1+\Re q_{\mu\boxplus\overline{\mu}}(iy))\notag\\
&=-\frac 1{2y}\int\limits_{\Bbb R}\frac {u^2}{u^2+y^2}\,(\mu\boxplus
\overline{\mu})(du)(1+\Re q_{\mu\boxplus\overline{\mu}}(iy)),
\end{align}
holds, where $\Re q_{\mu\boxplus\overline{\mu}}(iy)\to 0$ as $y\to\infty$.
We conclude from (\ref{4.3a}) and (\ref{4.4}) that
$$
\int\limits_{\Bbb R}\frac {u^2}{u^2+y_l^2}\,(\mu\boxplus\overline{\mu})(du)
=o(1/y_l^2),\quad l\to\infty,
$$ 
for $y_l:=1/t_l$. This relation implies $\int\limits_{\Bbb R}u^2
\,(\mu\boxplus\overline{\mu})(du)=0$ and therefore the~measure
$\mu\boxplus\overline{\mu}=\delta_0$.
Since $\phi_{\mu\boxplus
\overline{\mu}}(z)=\phi_{\mu}(z)+\phi_{\overline{\mu}}(z)=0$ for
$z\in\Gamma_{\alpha_1,\beta_1}$ with some $\alpha_1,\beta_1>0$, and $\Im\phi_{\mu}(z)\le 0$ and 
$\Im\phi_{\overline{\mu}}(z)\le 0$ for such $z$, we easily conclude 
that $\phi_{\mu}(z)=0,z\in\Gamma_{\alpha_1,\beta_1}$, and $\mu=\delta_a,\,a\in\mathbb R$, 
a~contradiction. The~lemma is proved.
\end{proof}

From the~definition of the~active exponent $\sigma_1(v)$ (see (\ref{3.7}), where
$K(z;v)$ and $\Lambda(z)$ are defined in (\ref{3.5}) 
with $Q(s)=Q_1(s)-Q_2(s)$, where $Q_1(s)$ and $Q_2(s)$ are distribution functions of
the~measures $\delta_{|a_1|}+\dots+\delta_{|a_n|}$ and $\delta_{|b_1|}+\dots+\delta_{|b_n|}$,
respectively, and $v(t):=t\Im \phi_{\mu}(i/t)$) we conclude
that $\sigma_1(v),\, 0<\sigma_1(v)\le 2$, is a~root of the~function $\Lambda_1(z)$. 

By the~assumptions of the~theorem and Lemma~\ref{3.1lem} it follows
that $\sigma_1(v)=2$.
Consider the~function $v_1(t):=v(t)-a_2t^2$, where we have chosen $a_2:=a_2(v)(2+\lambda)^2$. 
The~coefficient $a_2(v)$ and the~parameter $\lambda$ were chosen in Section~4.
It is clear that $v_1(t)$
is a~solution of equation (\ref{4.2}). Moreover
\begin{align}
K(z;v_1)&:=K(z;v)-a_2\int\limits_0^1\frac{s^{-z}-s^2}{z+2}\,dQ(s)\notag\\
&=K(z;v)-a_2\frac{\Lambda_1(-z)-\Lambda_1(2)}{z+2}=K(z;v)-a_2\frac{\Lambda_1(-z)}{z+2}.\notag
\end{align} 
Therefore
\begin{align}\label{4.5}
Res_{z=-2}\Big(t^{\lambda-z}&\frac{K(z;v_1)}{(z-\lambda)^2\Lambda_1(-z)}\Big)\notag\\
&=Res_{z=-2}\Big(t^{\lambda-z}\frac{K(z;v)}{(z-\lambda)^2\Lambda_1(-z)}\Big)
-a_2Res_{z=-2}\Big(\frac{t^{\lambda-z}}{(z-\lambda)^2(z+2)}\Big)\notag\\
&=t^{\lambda+2}\Big(a_2(v)-\frac{a_2}{(2+\lambda)^2}\Big)=0.
\end{align}

Thus, we may choose $a_2$ in such a~way 
that 2 is not an~active exponent of the~solution $v_1(t)$. Hence $v_1(t)$
has no active positive exponents. 

Hence we arrive at two cases. In the~first case $v_1(t)\ne 0$ in some interval $(0,t_0)$ with
$0<t_0\le 1$. In the~second case there exists a~sequence $\{t_k\},\,
0<t_k\le 1,\,\lim_{k\to\infty}t_k=0$, such that $v_1(t_k)=0$.

In the~first case, by Lemma~{\ref{3.1lem}, there exists a~positive
active exponent of the solution $v_1(t)$, a~contradiction. Hence, we may consider
the~second case only. In this case it is easy to see that
the~function $\phi_{\mu\boxplus\overline{\mu}}(z)$ satisfies the~assumptions of 
Lemma~\ref{3.5l} and we obtain that $\mu\boxplus\overline{\mu}$ is a semicuclar measure.
 
In order to complete
the~proof of the~sufficiency of the~assumptions of the~theorem 
it remains to apply the~following lemma.
\begin{lemma}\label{4.2lem}
Assume that the~function $\Lambda_2(z)$ satisfies the~condition: $\Lambda_2(2k-1)\ne 0$
for all $k=2,3,\dots$.
Let the~statistics $L_1$ and $L_2$ be identically distributed and let
$\mu\boxplus\overline{\mu}$ be a~semicircular measure.
Then $\mu$ is a~semicircular measure as well. 
\end{lemma}
\begin{proof}
By Lemma~\ref{3.6lem}, the~measure $\mu$ has a~compact support. Hence, 
the~Voiculescu transform $\phi_{\mu}(z)$ is 
an~analytic function in the~domain $|z|>R$ with some parameter $R>0$
and it admits in this domain the~following Laurent expansion
\begin{equation}\notag
\phi_{\mu}(z)=\kappa_1+\frac{\kappa_2}{z}+\sum_{l=3}^{\infty}
\frac{\kappa_l}{z^{l-1}}.
\end{equation}
Here $\kappa_2\ge 0$. Since $\mu\boxplus\overline{\mu}$
is a~semicircular measure, we have, using (\ref{2.3}),
\begin{equation}\notag
\sum_{l=1}^{\infty}\frac{\kappa_{2l}}{z^{2l-1}}=\frac bz,\quad|z|>R,
\end{equation}
where $b>0$. From this formula we deduce that $\kappa_{2}=b$ and $\kappa_{2l}=0$ for $l=2,3,\dots$.
Since the~function $\phi_{\mu}(z)$ satisfies the~equation~(\ref{4.1}),
we obtain the~relation
\begin{equation}\notag
\frac{\kappa_2\Lambda_2(2)}z+\sum_{l=1}^{\infty}
\frac{\kappa_{2l-1}\Lambda_2(2l-1)}{z^{2l-2}}=0,\quad |z|>R.
\end{equation}
By the~assumptions of the~lemma 
$\Lambda_2(2l-1)\ne 0$ for 
$l=2,3,\dots$, we conclude that $\kappa_{2l-1}=0$ for $l=2,3,\dots$. 

Thus, the~lemma is proved.
\end{proof}

{\it Necessity}. We note that in order that the~statement (1) of the~theorem
implies the~statement (2) it is necessary that $\Lambda_1(2)=0$.

We shall first assume that the~function $\Lambda_1(z)$ has a~root $\gamma_1$ 
such that $0<\gamma_1<2$. Let $0<\gamma_1<1$ or $1<\gamma_1<2$.   
By Lemma~\ref{5.1lem}, there exist a~symmetric 
probability measure $\mu$ whose the~Voiculescu transform has the~form $\phi_{\mu}(z)=-e^{i\gamma_1\pi/2}
z^{-\gamma_1+1}$. We conclude for this function that
\begin{align}
\sum_{j=1}^n a_j\phi_{\mu}(z/a_j)-\sum_{k=1}^n b_k\phi_{\mu}(z/b_k)
&=\sum_{j=1}^n |a_j|\phi_{\mu}(z/|a_j|)-\sum_{k=1}^n |b_k|\phi_{\mu}(z/|b_k|)\notag\\
&=-e^{i\gamma_1\pi/2}
z^{-\gamma_1+1}\Lambda_1(\gamma_1)=0,\quad z\in\mathbb C^+.\notag
\end{align}
Let $\gamma_1=1$. By Lemma~\ref{5.1lem}, there exist symmetric 
probability measure $\mu$ whose the~Voiculescu transform has the~form $\phi_{\mu}(z)=-i$. We obtain for 
this function
\begin{equation}\notag
\sum_{j=1}^n a_j\phi_{\mu}(z/a_j)-\sum_{k=1}^n b_k\phi_{\mu}(z/b_k)=-i\Lambda_1(1)=0,
\quad z\in\mathbb C^+.
\end{equation}

We shall now assume that $\gamma_1=2$ and $2$ is not a~simple root of the~function
$\Lambda_1(z)$. 
By Lemma~\ref{5.3lem}, 
there exist a~symmetric probability measure $\mu$ whose the~Voiculescu transform has the~form
\begin{equation}\notag
\phi_{\mu}(z)=\frac {1+\varepsilon (\log z-i\pi/2)}z,
\quad z\in \Bbb C^+,
\end{equation}
with sufficiently small parameter $\varepsilon>0$. It is easy to see that
\begin{equation}\notag
\sum_{j=1}^n a_j\phi_{\mu}(z/a_j)-\sum_{k=1}^n b_k\phi_{\mu}(z/b_k)=
\Lambda_1(2)\frac 1z+\frac{\varepsilon}z\sum_{s=0}^{1}(-1)^s
\Big(\log z-\frac{i\pi}2\Big)^s\Lambda_1^{(1-s)}(2)=0.
\end{equation}

Assume that $\gamma_1>2$ and $\gamma_1$ is not even.
By Lemma~\ref{5.2lem}, there exist a~symmetric probability measure $\mu$ whose the~Voiculescu transform 
$\phi_{\mu}(z)$ has the~form 
$$
\phi_{\mu}(z)=\frac 1z-\varepsilon \cos\big((\gamma_1-1)\pi/2\big)
\frac{ie^{i(\gamma_1-1)\pi/2}}{z^{\gamma_1-1}},
\quad z\in \Bbb C^+,
$$
with sufficiently small parameter $\varepsilon>0$.
We deduce as above that
\begin{equation}\notag
\sum_{j=1}^n a_j\phi_{\mu}(z/a_j)-\sum_{k=1}^n b_k\phi_{\mu}(z/b_k)=\Lambda_1(2)\,\frac 1z-
\varepsilon \cos\big((\gamma_1-1)\pi/2\big)\frac{ie^{i(\gamma_1-1)\pi/2}}{z^{\gamma_1-1}}
\Lambda_1(\gamma_1)=0
\end{equation}
for $z\in\mathbb C^+$.

We shall now show that if there exists a~positive integer $m>2$ such that
$\Lambda_2(m)=0$, then the~statement (1) of the~theorem does not imply the~statement (2).
Using Corollary~\ref{2.3acor} (see \cite{BeVo:1995} as well), consider 
a~probability measure $\mu$ with the~Voiculescu transform 
\begin{equation}\notag
\phi_{\mu}(z):=\frac 1z+\frac{\varepsilon}{z^{m-1}},
\end{equation}
where $\varepsilon\in\mathbb R$ and is sufficiently small by modulus. 
We easily see that the~function
$\phi_{\mu}(z)$ satisfies the~equation (\ref{4.1}). Moreover, the~probability measure
has a~compact support.

Thus, we have established that if 2 is not unique simple positive zero of the~function
$\Lambda_1(z)$ or there exist odd positive numbers $2l+1\ge 3$ such that $\Lambda_2(2l+1)=0$
the~statement (1) does not imply the~statement (2) of the~theorem.

The~theorem is completely proved.
$\square$

{\bf Proof of Theorem~$\ref{2.1th*}$}.
The proof of this theorem easily follows from the arguments that we used in the proof
of Theorem~$\ref{2.1th}$}. Therefore we omit it.
$\square$

{\bf Proof of Theorem~$\ref{2.2th}$}. We keep all previous notations.
We assume that the~statistics $L_1$ and $L_2$ are identically distributed.
By the~assumptions of the~theorem, $m_{2s}(\mu)<\infty$ with $s:=[\gamma/2+1]$,
where $\gamma$ is maximum of the~real parts of zeros of the~function $\Lambda_1(z)$.
By Proposition~\ref{3.5pr} and (\ref{4.1*}), we have
\begin{equation}\notag
v(t):=t\Im\phi_{\mu}(i/t)=-\kappa_2(\mu)t^2+\dots+(-1)^s\kappa_{2s}(\mu)t^{2s}+o(t^{2s}),
\quad t\to +0.
\end{equation}
Therefore $\lim_{t\to +0}v^{(2s)}(t)=(-1)^s(2s)!\kappa_{2s}(\mu)$. We now conclude from 
Lemmas~\ref{3.1lem} and \ref{3.3lem}
that all active exponents of $v(t)$ are positive integers and simultaneously simple exponents.
Since the~number of active exponents of $v(t)$ is finite we can use the~formula (\ref{3.7a}).
Using this identity we easily obtain the~relation
\begin{equation}\label{4.6}
\frac 12\Im\phi_{\mu\boxplus\bar{\mu}}(i/t)=\Im\phi_{\mu}(i/t)=\sum_{l=1}^{2s}b_lt^{l-1},\quad 0<t<1,
\end{equation}
where $b_l,\,l=1,\dots,2s$, are real coefficients. We deduce from (\ref{4.6})
that $b_l=\frac 12\kappa_l(\mu\boxplus\bar{\mu}),\,l=1,\dots,2s$, and 
$\kappa_l(\mu\boxplus\bar{\mu})=0$ for $l\ge 2s+1$. The~function $\phi_{\mu}(z)$
satisfies the~equation (\ref{4.1}). Therefore, using (\ref{4.6}), we get
\begin{equation}\label{4.7}
\sum_{l=1}^{2s}\Lambda_2(l) \frac{\kappa_l(\mu\boxplus\bar{\mu})}{z^{l-1}}=0,\quad
z\in\mathbb C^+. 
\end{equation}
We conclude from (\ref{4.7}) that $\kappa_{l}(\mu\boxplus\bar{\mu})=0$ for $l=2,\dots$.
Thus, $\mu\boxplus\bar{\mu}$ is a~semicircular measure.

From the~assumption of the~theorem and Lemma~\ref{4.2lem} it follows
that $\mu$ is a~semicircular measure as well.

One can prove the~necessity of the~assumptions of Theorem~\ref{2.2th} in the~same way
as in the~proof of the~necessity of the~assumptions of Theorem~\ref{2.1th}.

Thus, the~theorem is completely proved.
$\square$

{\bf Proof of Corollary~$\ref{2.1co}$}.
Since $|a_j|\le 1,\,j=1,\dots,n$, and $\Lambda_1(2)=0$, we note that
\begin{equation}\notag
|a_1|^x+\dots+|a_n|^x>1\quad\text{for}\quad 0<x<2\quad\text{and}\quad
|a_1|^x+\dots+|a_n|^x<1\quad\text{for}\quad x>2.
\end{equation}
Thus, $\Lambda_1(x)\ne 0$ for all positive $x\ne 2$. In addition, it is easy to see that
$\Lambda_1'(2)\ne 0$. By the~inequality
\begin{equation}\notag
|a_1^m+\dots+a_n^m|\le |a_1|^m+\dots+|a_n|^m<1 \quad\text{for}\quad m=3,\dots,
\end{equation}
we have $\Lambda_2(m) \ne 0,\, m=3,\dots $. By Theorem~\ref{2.1th}, if $\mu$ is a~non-degenerate
probability measure, then $\mu$ is a~semicircular measure. If $\mu=\delta_a$ with some $a\in\mathbb R$,
then (\ref{4.1}) holds if $\Lambda_2(1)=0$. If $\Lambda_2(1)\ne 0$, then (\ref{4.1}) holds for $\mu=\delta_0$
only.
$\square$


\begin{thebibliography}{99}
\itemsep=\smallskipamount

\bibitem{Akh:1965} Akhiezer,~N. I. 
{\em The~classical moment problem and some related questions 
in analysis.}
Hafner, New York (1965).



\bibitem{AkhG:1963} Akhiezer,~N. I. and Glazman,~I. M. 
{\em Theory of Linear Operators in Hilbert Space.}
Ungar, New York (1963).


\bibitem{BelBe:2007} Belinschi,~S. T., Bercovici~H.
\emph{A new approach to subordination results in free probability.}
J. Anal. Math. {\bf 101}, 357--365 (2007).


\bibitem{BeVo:1993} Bercovici,~H., and Voiculescu,~D.
\emph{Free convolution of measures with unbounded support.}
Indiana Univ. Math. J., {\bf 42}, 733--773 (1993).


\bibitem{BeVo:1995} Bercovici,~H., and Voiculescu,~D.
\emph{Superconvergence to the~central limit and failure
of the~Cram\'er theorem for free random variables.}
Probab. Theory Relat. Fields,
{\bf 102}, 215--222 (1995).




\bibitem{BeP:1999} Bercovici,~H., and Pata, V.
\emph{Stable laws and domains of attraction in free
probability theory $($with an appendix by Ph. Biane$)$.}
Annals of Math., {\bf 149}, 1023--1060, (1999).



\bibitem{Bi:1998} Biane,~Ph.
\emph{Processes with free increments.}
Math. Z., 143--174 (1998).


\bibitem{ChG:2005} Chistyakov,~G. P. and G\"otze,~F.
\emph{The~arithmetic of distributions in free probability theory.}
Cent. Eur. J. Math., 9(5), 997--1050 (2011).
DOI: 10.2478/s11533-011-0049-4,
ArXiv:math/0508245.

\bibitem{ChG:2006} Chistyakov,~G. P. and G\"otze,~F.
\emph{Limit theorems  in free probability theory. I}
Anal. Probab. {\bf 36}, No 1, 54--90 (2008).

\bibitem{Dar:1953} Darmois,~G.
\emph{Analyse g\'en\'erale des liaisons stochastiques.}
Rev. Inst. Internationale Statist.,
{\bf 21}, 2--8 (1953).



\bibitem{HiPe:2000} Hiai,~F. and Petz,~D. 
{\em The Semicircle Law, Free Random Variables and Entropy.}
American Mathematical Society, (2000).
 

\bibitem{ILin:1975} Ibragimov,~I. A. and Linnik,~Yu. V.
{\em Independent and Stationary Sequences of Random Variables.}
Wolters--Noordhoff, Groningen.

\bibitem{KLR:1973} Kagan, A. M., Linnik, Yu. V., Rao, C. R.  
{\em Characterization problems in mathematical statistic.}
John Wiley \& Sons,  New York, London, Sydney, Toronto (1973).

\bibitem{Ka:2007} Kargin, V.
\emph{On superconvergence of sums of free randim variables.}
Annal. Probab., {\bf 35}, 1931--1949 (2007).




\bibitem{KrN:1977} Krein, M. G., and Nudel'man, A. A.  
{\em The Markov moment problem and extremal problems.}
Amer. Math. Soc., Providence, Rhode Island (1977).



\bibitem{Leh:2003} Lehner, F.   
\emph{Cumulants in noncommutative probability theory. II Generalized Gaussian random variables.}
Probab. Theory Relat. Fields, {\bf 127}, no 3, 407--422 (2003).

\bibitem{Leh:2004} Lehner, F.   
\emph{Cumulants in noncommutative probability theory. I Noncommutative exchageability systems.}
Math. Z., {\bf 248}, no 1, 67--100 (2004).



\bibitem{Li:1963} Linnik,~Yu., V.
{\em Linear forms and statistical criteria. I.}
Selected Transl. Math. Statist. and Prob., {\bf 3}, 
Amer. Math. Soc., Providence, R. I. {\bf 3}, 1--40 (1963).

\bibitem{Li:1962} Linnik,~Yu., V.
{\em Linear forms and statistical criteria. II.}
Selected Transl. Math. Statist. and Prob., {\bf 3}, 
Amer. Math. Soc., Providence, R. I. {\bf 3}, 41--90 (1962).


\bibitem{LiO:1977} Linnik,~Yu. V. and Ostrovskii,~I. V.
{\em Decomposition of Random Variables and Vectors.}
Amer. Math. Soc., Providence, Rhode Island (1977).


\bibitem{Ma:1992} Maassen,~H.
\emph{Addition of Freely Independent Random Variables.}
Journal of functional analysis, {\bf 106}, 409--438 (1992).

\bibitem{Mar:1938} Marcinkiewicz,~J.
{\em Sur une propriete de la loi de Gauss.}
Math. Zeitschrift, {\bf 44}, 622--638 (1938).

\bibitem{Ni:1996} Nica,~A.
\emph{$R$-transforms of free joint distributions and non-crossing partitions.}
J. Funct. Anal., {\bf 135}, 271--296 (1996).


\bibitem{NiSp:2006} Nica,~A. and Speicher,~R.
{\em Lectures on the Combinatorics of Free Probability.}
Cambridge University Press, (2006).

\bibitem{Po:1923} Polya,~G.
{\em Herleitung des Gauss'schen Fehlergesetzes aus einer Funktionalgleichung.}
Math. Zeitschrift, {\bf 18}, 185--188 (1923).

\bibitem{Sk:1953} Skitovich,~V. P. 
\emph{On a~property of the~normal distribution.}
DAN SSSR, {\bf 89}, 217--219 (1953).

\bibitem{Sk:1954} Skitovich,~V. P. 
\emph{Linear forms in independent random variables and the~normal distribution law.}
Izvestiia AN SSSR, Ser. Matem., {\bf 18}, 185--200 (1954). 
 




\bibitem{Vo:1986} Voiculescu,~D.V.
\emph{ Addition of certain noncommuting random variables.}
J. Funct. Anal., {\bf 66}, 323--346 (1986).


\bibitem{Vo:1987} Voiculescu,~D.V.
\emph{Multiplication  of certain noncommuting random variables.}
J. Operator Theory, {\bf 18}, 223--235 (1987). 


\bibitem{Vo:1992} Voiculesku,~D., Dykema,~K., and Nica,~A. 
{\em Free random variables.}
CRM Monograph Series, No 1, A.M.S., Providence, RI (1992).

\bibitem{Vo:1993} Voiculescu,~D.V.
\emph{ The~analogues of entropy and Fisher's information \tc{measure}
in free probability theory. I.}
Comm. Math. Phys., {\bf 155}, 71--92 (1993).



\bibitem{ZYa:1991} Zinger, A. A. and Yanuschkyavichyus, R. V.
\emph{On probabilistic solutions of some functional equations.}
J. Soviet Math., {\bf 57}, no 4, 3225--3233 (1991). 

\end{thebibliography}
\end{document}